\documentclass[11pt, reqno, oneside, notitlepage]{amsart}

\usepackage[a4paper, total={6in, 9.1in}]{geometry}

\usepackage{amsmath,amscd}
\usepackage{amssymb}
\usepackage{amsthm}
\usepackage{comment}
\usepackage{graphicx}
\usepackage{epstopdf} 
\usepackage{mathrsfs}

\usepackage{bm}
\usepackage{hyperref}
\usepackage{xcolor}
\hypersetup{
	colorlinks,
	linkcolor={blue},
	urlcolor={blue},
	citecolor={red}
}

\theoremstyle{plain}
\newtheorem{thm}{Theorem}[section]
\newtheorem{prop}{Proposition}[section]

\newtheorem{lem}[prop]{Lemma}
\newtheorem{cor}[prop]{Corollary}

\newtheorem{defi}[prop]{Definition}
\newtheorem{rmk}[prop]{Remark}

\newtheorem*{proposition*}{Proposition}

\numberwithin{equation}{section}
\newcommand {\R} {\mathbb{R}} 
 \newcommand {\N} {\mathbb{N}}
\newcommand {\p} {\partial}

\newcommand{\eps}{\epsilon}

\newcommand {\supp} {\mathrm{supp}}

\newcommand{\wt}{\widetilde}

\newcommand{\LC}{\left(}
\newcommand{\RC}{\right)}

\newcommand{\norm}[1]{\lVert #1 \rVert}         

\DeclareMathOperator {\dist} {dist}

\title[Simultaneous recoveries for semilinear parabolic systems]{Simultaneous recoveries for semilinear parabolic systems}

\author[Y.-H. Lin]{Yi-Hsuan Lin}
\address{Department of Applied Mathematics, National Yang Ming Chiao Tung University, Hsinchu 30050, Taiwan}
\email{yihsuanlin3@gmail.com}

\author[H. Liu]{Hongyu Liu}
\address{Department of Mathematics, City University of Hong Kong, Kowloon, Hong Kong SAR, China}
\email{hongyu.liuip@gmail.com, hongyliu@cityu.edu.hk}

\author[X. Liu]{Xu Liu}
\address{Key Laboratory of Applied Statistics of MOE,  
	School of Mathematics and Statistics,  Northeast Normal University, Changchun, China}
\email{liux216@nenu.edu.cn}

\author[S. Zhang]{Shen Zhang}
\address{Department of Mathematics, City University of Hong Kong, Kowloon, Hong Kong SAR, China}
\email{szhang347-c@my.cityu.edu.hk}

\begin{document}
	\maketitle
	
	\begin{abstract}
	
	In this paper, we study inverse boundary problems associated with semilinear parabolic systems in several scenarios where both the nonlinearities and the initial data can be unknown. We establish several simultaneous recovery results showing that the passive or active boundary Dirichlet-to-Neumann operators can uniquely recover both of the unknowns, even stably in a certain case. It turns out that the nonlinearities play a critical role in deriving these recovery results. If the nonlinear term belongs to a general $C^1$ class but fulfilling a certain growth condition, the recovery results are established by the control approach via Carleman estimates. If the nonlinear term belongs to an analytic class, the recovery results are established through successive linearization in combination with special CGO (Complex Geometrical Optics) solutions for the parabolic system.

		\medskip

		
		
	\end{abstract}

	\tableofcontents

	\section{Introduction}\label{Sec 1}

\subsection{Mathematical setup and statement of the main results}

In this paper, we are concerned with inverse problems for semilinear parabolic equations. Depending on the form of the nonlinear term, there are two setups for our study, which shall be discussed separately in what follows. 

First, we consider the case that the nonlinear term belongs to a $C^1$ class fulfilling a certain growth condition. We begin by introducing the forward model. 
Let $\Omega\subset\mathbb{R}^n$ be a   bounded  domain
with a $C^\infty$-smooth boundary $\Gamma$ for $n \in \N$ and 
$\Gamma_0$ be a nonempty relatively open subset of $\Gamma$.  For  any  $T>0$,   
we set 
$
Q=\Omega\times(0,  T)\text{ and }\Sigma=\Gamma \times(0,  T).
$
Assume that
$\gamma=\LC \gamma_{ij}(x,t)\RC_{i,j=1}^n \in C^{2,1}(\overline{Q}; \R^{n\times n})$
is  a symmetric  matrix-valued function in $\overline{Q}$, such that 
\[
\rho_0|\xi|^2 \leq \sum_{i,j=1}^n \gamma_{ij}(x,t)\xi_i \xi _j \leq \rho_0^{-1}|\xi|^2, \quad
\forall\  (x,t)\in \overline{Q }\ \ \text{ and }\ \ \xi=(\xi_1,\ldots, \xi_n) \in \R^n,
\]
for some positive constant $\rho_0\in (0, 1)$. Moreover,  
we denote by $H^{s, r}(Q)$,  $H^{s, r}(\Gamma)$,  $C^{k+\alpha}(\overline{\Omega})$ and
 $C^{k+\alpha, \frac{k}{2}+\frac{\alpha}{2}}(\overline{Q})$, respectively,   the standard Sobolev  spaces and 
 H\"older spaces for $s, r\in\mathbb R$, $k\in  \mathbb N$ and $\alpha\in(0, 1)$.
We refer to \cite{Adams2003sobolev} and \cite{evans1998partial} for details of these Banach spaces. Consider the following  semilinear parabolic equation: 
\begin{eqnarray}\label{eq:parabolic1}
	\begin{cases}
		u_{t}-\nabla\cdot(\gamma\nabla u)+a(x,t,u)=0  &\text{ in }\ Q,\\
		u=f   &\text{ on }\ \Sigma,\\
		u(x, 0)=g(x)  &\text{ in }\ \Omega,
	\end{cases}
\end{eqnarray}
where $u_t=\partial_t  u= \frac{\p u}{\p t}$, $g \in H^1_0(\Omega)$,   
$f\in L^2(\Sigma)$ and $a=a(x,t,y): 
Q\times\mathbb R\rightarrow \mathbb R$ is a given  function,  satisfying 
suitable conditions that will be specified later.

For any  $g \in H^1_0(\Omega)$  and a suitable function   $a: Q\times \R\to \R$, 
which guarantees the global well-posedness of (\ref{eq:parabolic1}) (see Section \ref{Sec 2}),  
we introduce the following  Dirichlet-to-Neumann (DN for short) operator:
\begin{align}\label{eq:DN1}
	\begin{split}
		\Lambda_{a,g}: \mathcal E & \rightarrow  L^2(\Gamma_0\times(0, T)) , \\  
		f&\mapsto  \partial_\nu u_{f}\Big|_{\Gamma_0\times(0, T)}. 
	\end{split}
\end{align}
Here we use the notation that $\p_\nu u(x):=\nabla u(x)\cdot \nu(x), \text{ for }x\in \Gamma$,
 where  $\nu=(\nu_1,\ldots, \nu_n)\in\mathbb{S}^{n-1}$ signifies the exterior unit normal to $\Gamma$,  and 
$u_f$ is the solution to \eqref{eq:parabolic1}  associated to the initial 
data $g\in H^1_0(\Omega)$ and  boundary data $f\in\mathcal E$ with
\begin{eqnarray*}
	&&\mathcal{E}=\Big\{  f\in    L^2(\Sigma)\  
	\Big|\    (\ref{eq:parabolic1}) \mbox{ is well-posed associated to }  g \mbox{ and }a,  \mbox{ such that }
	\\
	&&\quad\quad\quad\quad\quad\quad\quad\quad \, u\in  C([0, T]; L^2(\Omega))
	\mbox{   and  }  \left. \partial_\nu u \right|_{\Gamma_0\times(0, T)} \in L^2(\Gamma_0\times(0, T))
	\Big\}.
\end{eqnarray*}
It is known that when $a\in L^\infty(Q; W^{1, \infty}(\mathbb R))$,
$$
\Big\{  f\in    H^{\frac{3}{2}, \frac{3}{4}}(\Sigma)\  
\Big|\   
f(x, 0)=0 \mbox{ on  }\Gamma
\Big\}\subseteq \mathcal{E}.
$$
When $f\equiv0$  in $Q$,  the DN map is simply denoted by 
\begin{align*}
	\Lambda_{a,g}^{(0)}:=\Lambda_{a,g}(0),
\end{align*}
In such a case and in the physical situation, the field $u$ is generated by the initial data $g$, acting as a source, which is assumed to be unknown in our inverse problem study. Hence, the boundary measurement encoded in $\Lambda_{a,g}^{(0)}$ is passively taken by the observer, and in the literature, $\Lambda_{a,g}^{(0)}$ is usually referred to as the {\it passive measurement}. In contrast, $\Lambda_{a,g}(f)$ associated with a nontrivial boundary input $f$ is called the {\it active measurement} since the field $u$ is actively induced by the observer by imposing a boundary input. 

Associated to the forward model \eqref{eq:parabolic1}--\eqref{eq:DN1}, we are interested in the following two inverse problems:
\begin{itemize}
	\item \textbf{Inverse Problem 1.} Can we identify the  unknown functions $(a,g)$ by using the passive measurement $\Lambda_{a,g}^0$?
	
	\item  \textbf{Inverse Problem 2.} Can we identify the unknown functions $(a,g)$ by using the active measurement $\Lambda_{a,g}$?
\end{itemize}

It is emphasized that the principal coefficient $\gamma=\gamma(x,t)$ of our \textbf{Inverse Problem 1} and \textbf{Inverse Problem 2} can be space-time dependent. Next, we need to introduce certain a priori conditions on the nonlinear term $a$ in order to guarantee the well-posedness of the forward problem as well as the feasibility of the inverse problems. To that end, we assume that $a: Q\times\mathbb R\rightarrow\mathbb R$ 
with $a(x, t, \cdot)\in C^1(\mathbb R)$ in  $Q$, and satisfies the following growth condition: 
\begin{align}\label{condition of nonlinear f at infinity data}
\limsup\limits_{y\rightarrow\infty}\displaystyle\frac{\p_y a(x, t, y)}{\mbox{ln}^{\frac{1}{2}}|y|}=0,\quad
	\mbox{ uniformly for }(x,  t)\in  Q.
\end{align}
It is clear that any function in $L^\infty(Q; W^{1, \infty}(\mathbb R))$ satisfies 
the condition (\ref{condition of nonlinear f at infinity data}).
For notational clarity, we set
\begin{eqnarray}\label{set A}
	\begin{array}{ll}
		\displaystyle\mathcal{A}_T=\Big\{ a: Q\times\mathbb R\rightarrow\mathbb R\ \Big|  &
		a(x, t, \cdot)\in C^1(\mathbb R)\mbox{ in }Q, \  a(\cdot, \cdot, 0)\in  L^2(Q),\\
		&\displaystyle\mbox{ and the  condition } 
		(\ref{condition of nonlinear f at infinity data})\mbox{ is fulfilled } \Big\}.
	\end{array}
\end{eqnarray}  
In Section \ref{Sec 2}, we shall show that for any $g \in H^1_0(\Omega)$,
$a\in \mathcal{A}_T$ and $f=0$, \eqref{eq:parabolic1} has a unique solution $u\in H^{2, 1}(Q)$ and therefore, $\partial_\nu u \in L^2(\Sigma).$

We are in a position to state the first recovery result for the inverse problems introduced above. 

\begin{thm}[Conditional stability of determining initial data by passive measurement]\label{Main Thm 1} 
	Given  $a \in\mathcal A_T$  and 
	$g_j \in H_0^1(\Omega)$ $(j=1, 2)$.   Let $\Lambda_{a,g_j}^0$  be the
	$($\mbox{passive}$)$ DN map associated to the following semilinear parabolic equation: 
	\begin{align}\label{IBVP for thm 1 for j=1,2}
		\begin{cases}
			\partial_t u_{j}-\nabla\cdot(\gamma\nabla u_j)+a(x,t,u_j)=0  &\text{ in }\ Q,\\
			u_j=0  &\text{ on }\ \Sigma,\\
			u_j(x, 0)=g_j(x), 
			&\text{ in }\ \Omega.
		\end{cases}
	\end{align}
	For any $M>0$,  if 
	\[
	\norm{g_1-g_2}_{H^1_0(\Omega)}\leq M,  
	\]
	then there exist  positive constants $C$  and $\delta_0\in (0, 1)$, 
	depending only on $n,  T$  and $\Omega$, such that 
	the following quantitative stability estimate holds:
	\begin{align}\label{Stability estimate in Thm 1}
		\begin{split}
			\norm{g_1 -g_2}^2_{ L^2(\Omega)} 
			\leq & 
			\frac{C(1+M)}{\delta_0}\norm{\Lambda^0_{a, g_1}-\Lambda^0_{a, g_2}}_{L^2(\Gamma_0\times(0, T))}\\
			&-\frac{CM^2}{\ln\LC \delta_0 \norm{\Lambda^0_{a, g_1}-\Lambda^0_{a, g_2}}
				_{L^2(\Gamma_0\times(0, T))}\RC }.
		\end{split}
	\end{align}
\end{thm}
By Theorem \ref{Main Thm 1}, it is directly verified that if $\Lambda^0_{a, g_1}=\Lambda^0_{a, g_2}$ on $\Gamma_0\times (0, T), then $
$g_1=g_2$ in $\Omega$. Theorem~\ref{Main Thm 1} partially answers {\bf Inverse Problem 1} that if the nonlinear term $a$ belongs to the general class \eqref{set A} and is a-priori known, then the initial data $g$ can be uniquely recovered (in a stable manner) by the passive measurement. 

We proceed to consider \textbf{Inverse Problem 2} and introduce
another ``admissible" set on $a$:
\begin{eqnarray}\label{set C}
	\begin{split}
		\mathcal{B}_{T}=\Big\{ a:  Q\times\mathbb R\rightarrow\mathbb R\  \Big|\      
		&a(x, t, y)=a_0(x, t, y)\chi_{[0, T-\epsilon]}(t)+c(x,t,y)\chi_{[T-\epsilon, T]}(t) \\[2mm]
		&\mbox{for some }\epsilon>0\mbox{ and  }a_0\in\mathcal{A}_T, \\[2mm]
		&\mbox{where } 
		c\in\mathcal{A}_T\mbox{ and }c(x, t, 0)=0 \mbox{ in } Q \Big\},  
	\end{split} 
\end{eqnarray}
where $\chi_E(t)=\begin{cases}
	1  &t\in E,\\
	0 &\text{otherwise}
\end{cases}$ signifies the characteristic function of a set $E\subseteq\mathbb R$.

Our main unique recovery result for \textbf{Inverse Problem 2}  is stated as follows.

\begin{thm}[Uniqueness of determining initial data by active measurements]\label{Main Thm 2}
	Given $a_j\in \mathcal B_{T}$ 
	and $g_j\in H^1_0(\Omega)$ $(j=1, 2)$.  
	Let $\Lambda_{a_j,g_j}$  be the $($\mbox{active}$)$ DN map of the
	semilinear  parabolic equation: 
	\begin{align}\label{IBVP for thm 2 for j=1,2}
		\begin{cases}
			u_{j, t}-\nabla\cdot(\gamma\nabla u_j)+a_j(x,t,u_j)=0  &\text{ in } Q,\\
			u_j=f  &\text{ on } \Sigma,\\
			u_j(x, 0)=g_j(x) ,  &\text{ in }\Omega.
		\end{cases}
	\end{align}
	If for any $f\in \mathcal  E$ with $\supp(f)\subseteq \Gamma_0\times [0, T]$,
	\begin{equation}\label{eq:a1}
		\Lambda_{a_1,g_1}(f)=\Lambda_{a_2,g_2}(f) \quad \text{ on }\quad \Gamma_0 \times (0,T),
	\end{equation}
	then one has that
	\begin{equation}\label{eq:a2}
		g_1 =g_2 \quad\mbox{  in  }\ \Omega.
	\end{equation}	
\end{thm}
Theorem  \ref{Main Thm 2} means  that 
the map $\Lambda_{a,g}$ uniquely determines the initial data $g$, 
independent of  functions $a\in \mathcal B_{T}$. 

In the second setup of our study, we consider the case that the nonlinear term $a$ belongs to an analytic class. In such a case, we can assume that both the initial data and the nonlinear term are unknown and can simultaneously recover both of them. To that end, we first introduce the analytic class for the nonlinear term.
  
\begin{defi}[Admissible class]\label{Def: admissible coefficients}
	Let $b=b(x,t,y):\overline{Q}\times\R  \rightarrow \R$ satisfy the following conditions:
		\begin{align}\label{condition of the nonlinear term f(x,t,z)}
			\begin{cases}
				\mbox{the map $ y\mapsto b(\cdot,\cdot,  y)$ is analytic  on $\mathbb  R$ with values in $C^{2+\alpha,1+\alpha/2}(\overline{Q})$},  \\
				b(x,t,0)=0 \mbox{ in } Q,
			\end{cases}
		\end{align}
	    for some $\alpha\in (0, 1)$.
		This means that  $b$ can be written  as the Taylor expansion at any $y_0\in 
		\R$:
		\begin{equation}\label{eq:source1}
			b(x,t,y)= \sum_{k=0}^\infty b^{(k)}(x,t,y_0) \frac{(y-y_0)^k}{k!},
		\end{equation} 
		where $\dfrac{b^{(k)}(x,t,y_0)}{k!}:=\dfrac{\partial_y^k b(x,t,y_0)}{k!}$ are Taylor's coefficients at $y_0\in \R$ for any $k\in \N$.	
\end{defi}

Next, let $\Omega \subset \R^n$ be a bounded domain with a $C^\infty$-smooth boundary $\Gamma$, for $n \geq 2$, and $T>0$. We introduce the forward model by considering the following semilinear parabolic equation:  
\begin{align}\label{eq:parabolic simul}
	\begin{cases}
		\partial_t u- \Delta u +b(x,t,u)=0  &\text{ in }\ Q,\\
		u= f   &\text{ on }\ \Sigma,\\
		u(x,  0)=g(x),     &\text{ in }\ \Omega,
	\end{cases}
\end{align}
where $b$ is the function given as Definition \ref{Def: admissible coefficients}. It is easily seen that the second condition \eqref{condition of the nonlinear term f(x,t,z)} of $b$ implies that $u=0$ is a trivial solution when the initial and boundary data are both zero.
In Section~\ref{Sec 2}, we shall prove the (local) well-posedness of the forward problem \eqref{eq:parabolic simul} under the assumption that the coefficient $b$, initial data $g$ and the boundary data $f$ fulfil the following compatibility condition:
\begin{align}\label{compatibility conditions semilinear}
	g(\cdot)=g_{x_i}(\cdot)=g_{x_i x_j}(\cdot)=f(\cdot, 0)=
	f_t(\cdot, 0)=0\quad\mbox{ on  }\quad \Gamma,  \quad \mbox{  for }i,  j=1, \cdots, n.
\end{align}

Next, we introduce the boundary measurement associated with \eqref{eq:parabolic simul} for our inverse problem study. Let $\mathbb{S}^{n-1}$ be the unit sphere of $\R^n$ and fix $\omega_0\in\mathbb{S}^{n-1}$. Define 
\begin{align}\label{Gamma plus}
	\Gamma_{\pm, \omega_0}=\Big\{x\in\Gamma\ \Big| \   \pm \nu(x)\cdot\omega_0\geq0\Big\}\quad \mbox{and}\quad	\Sigma_{\pm, \omega_0}=\Gamma_{\pm, \omega_0}\times(0,T).
\end{align}
Let $\mathcal{U}_{\pm}$ be a neighborhood of $\Gamma_{\pm,\omega_0}$ in $\Gamma$ and set
$$\mathcal{V}_+=\mathcal{U}_{+}\times (0,T)\quad\mbox{ and }\quad\mathcal{V}_{-}=\mathcal{U}_{-}\times (0,T).$$
With these notations and the local well-posedness at hand, the partial DN map $\Lambda^{\mathrm{P}}_{b,g}$
  is defined as:
\begin{align}\label{eq:DN2}
	\begin{split}
		\Lambda^{\mathrm{P}}_{b,g}:  C_0^{2+\alpha,1+\alpha/2}(\mathcal{V}_+) & \rightarrow  C^{1+\alpha,1+\alpha/2}(\mathcal{V}_{-}), \\
		f&\mapsto  \p_\nu u_f\Big|_{\mathcal{V}_-}, 
	\end{split}
\end{align}
for sufficiently small $f\in C_0^{2+\alpha,1+\frac{\alpha}{2}}(\mathcal{V}_+)$  
and $g\in C_0^{2+\alpha}(\Omega)$,
 which satisfy the compatibility conditions \eqref{compatibility conditions semilinear}, where $u_f$ is the unique solution to \eqref{eq:parabolic simul}. Meanwhile, with the (local) well-posedness at hand, the (full) DN map of the initial-boundary value problem \eqref{eq:parabolic simul linear} is given via 
 \begin{align}\label{eq:DN3}
 	\begin{split}
 		\Lambda_{b,g}:  C_0^{2+\alpha,1+\frac{\alpha}{2}}(\Sigma) & \rightarrow  C^{1+\alpha,1+\frac{\alpha}{2}}(\Sigma), \\
 		f&\mapsto  \p_\nu u_f\Big|_{\Gamma}, 
 	\end{split}
 \end{align}
 for sufficiently small $f\in C_0^{2+\alpha,1+\frac{\alpha}{2}}(\Gamma)$  
 and $g\in C_0^{2+\alpha}(\Omega)$.
 
Our third inverse problem is to ask:
\begin{itemize}
	\item \textbf{Inverse Problem 3.} Can we determine the unknown functions $(b,g)$ by using active measurements, either $\Lambda_{b,g}$ or $\Lambda^{\mathrm{P}}_{b,g}$?
\end{itemize}
The main result established for {\bf Inverse Problem 3} is stated as follows. 

\begin{thm}[Simultaneous recovery for the semilinear parabolic equation]\label{Main Thm:Simultaneous}
	Let $\Omega\subset \R^n$ be a bounded connected domain with a $C^\infty$-smooth boundary $\Gamma$ for $n\geq 2$, 
	$\Gamma_{+,\omega_0}$ be the set given by \eqref{Gamma plus}  and $b_j$ $(j=1, 2)$ be admissible. 
	Then there  exists  a  $\delta>0$,  such that for  any  $g_j \in C_0^{2+\alpha}(\Omega)$ $(j=1, 2)$  
	with $\left\| g_j\right\|_{C^{2+\alpha}(\Omega)}<\delta/2$, denote by 
	$\Lambda_{b_j,g_j}$ and $\Lambda_{b_j,g_j}^{\mathrm{P}}$ 
	 the full and partial DN maps of the semilinear parabolic equation: 
	\begin{align}\label{IBVP of simultaneous recovery}
		\begin{cases}
			u_{t}-\Delta u +b_j(x,t,u)=0  &\text{ in }\ Q,\\
			u= f   &\text{ on }\ \Sigma,\\
			u(x,  0)=g_j(x),   &\text{ in }\ \Omega,
		\end{cases}
	\end{align}
    for $j=1,2$, respectively. Then we have  the following results:
    \begin{itemize}
    	\item[(a)] (Full data) If
    	$$\Lambda_{  b_1,g_1}(f)=\Lambda_{ b_2,g_2}(f) \text{ on }\Sigma, $$
    	for any sufficiently small lateral boundary data $f\in C^{2+\alpha,1+\frac{\alpha}{2}}_0(\Sigma)$, then
    	$$
    	g_1=g_2  \text{ in } \Omega\quad \text{ and }\quad  b_1=b_2 \text{ in } Q\times \mathbb{R}.
    	$$
    	
    	\item[(b)] (Partial data) Given an open connected set $\Omega'\subset \Omega $ satisfying $\Gamma\subset\p\Omega'$, we assume that $b_1=b_2 $ in $\Omega'\times(0,T)\times \R$, if $\Lambda^{\mathrm{P}}_{b_j,g_j}$ are the partial DN map of the semilinear parabolic equation $\eqref{IBVP of simultaneous recovery}$ and $$\Lambda_{  b_1,g_1}^{\mathrm{P}}(f)=\Lambda^{\mathrm{P}}_{ b_2,g_2}(f) \text{ on }\mathcal{V}_-, $$
    	for any sufficiently small lateral boundary data $f\in C^{2+\alpha,1+\frac{\alpha}{2}}_0(\mathcal{V}_+)$, then 
    	$$
    	g_1=g_2  \text{ in } \Omega\quad \text{ and }\quad  b_1=b_2 \text{ in } Q\times \mathbb{R}.
    	$$
    \end{itemize}
\end{thm}

Theorem~\ref{Main Thm:Simultaneous} states that \textbf{Inverse Problem 3} can be solved under situations. In fact, we can determine $b(x,t,u)$ and $g(x)$ simultaneously by using active measurements with full data. Meanwhile, if we assume $b(x,t,u)$ is known a-priori in a small neighborhood of $\Sigma$, then we can also determine $b(x,t,u)$ and $g(x)$ simultaneously with partial measurements. 

\begin{rmk}
We would like to point out that 
\begin{itemize}
	\item[(a)] The proof of Theorem \ref{Main Thm:Simultaneous} relies on the successive linearization method combining with suitable \emph{complex geometrical optics} (CGO) solutions (see \cite{CY2018logarithmic} or Appendix \ref{Sec: Appendix}) and useful approximation properties (see Section \ref{Sec 4}). We can simply utilize either full or partial DN maps for the semilinear equation \eqref{IBVP of simultaneous recovery} in order to determine both coefficients and initial data uniquely. Moreover, the smallness assumption for both initial and boundary data is needed due to the local well-posedness of the forward problem \eqref{IBVP of simultaneous recovery} (see Section \ref{Sec 2}), but not used to solve the inverse problem.
	
	\item[(b)] In the statement (b) of Theorem \ref{Main Thm:Simultaneous}, the set $\Omega'$ can be chosen as $\Omega'=\Omega \setminus \overline{D}$ with $D\Subset \Omega$ such that $\Omega\setminus \overline{D}$ is connected. Moreover, such a set $D \Subset \Omega$ can be as large as possible so that the set $\Omega'$ is very ``thin". This means, for our partial data result, it is sufficient for us to know the coefficient near the boundary $\Gamma \times (0,T)$ a priori.
\end{itemize}
\end{rmk}

Finally, it would be interesting to consider the linear counterparts of the inverse problems studied in Theorem \ref{Main Thm:Simultaneous}, since to our best knowledge, the simultaneous recovery results are untouched in the literature even in the linear case, namely 
$$
b(x,t,y)=q(x,t)y
$$
as a linear function with respect to $y\in \R$. For this linear model, the smallness conditions for initial and boundary data are not required, since the well-posedness for general linear parabolic equations have been well understood (for example, see \cite[Chapter 7]{evans1998partial} or \cite{ladyzhenskaia1988linear}). To proceed, let us consider the linear parabolic equation:  
\begin{align}\label{eq:parabolic simul linear}
	\begin{cases}
	    \p _t u - \Delta u +q _j u=0  &\text{ in }\ Q,\\
		u= f   &\text{ on }\ \Sigma,\\
		u(x,  0)=g(x)     &\text{ in }\ \Omega.
	\end{cases}
\end{align}
In order to derive the well-posedness of  strong solutions to \eqref{eq:parabolic simul linear}, we need to impose the following compatibility condition:
\begin{align}\label{compatibility conditions intro}
	g(\cdot)=f(\cdot, 0)\quad\mbox{ on  } \Gamma. 
\end{align}
Via the condition \eqref{compatibility conditions intro}, one has the well-posedness of \eqref{eq:parabolic simul linear} immediately (see \cite[Chapter 7]{evans1998partial}). Therefore, one is able to define the corresponding partial DN map 
\begin{align}\label{eq:DN linear}
	\begin{split}
		\Lambda^{\mathrm{P}}_{q,g}:  C^{2+\alpha,1+\alpha/2}_0(\mathcal{V}_+) & \rightarrow C^{1+\alpha,1+\alpha/2}(\mathcal{V}_{-}) , \\  
		f&\mapsto  \p_\nu u \Big|_{\mathcal{V}_-},
	\end{split}
\end{align}
and the (full) DN map 
\begin{align}\label{eq:DN linear_full}
	\begin{split}
		\Lambda_{q,g}:  C^{2+\alpha,1+\alpha/2}_0(\Sigma) & \rightarrow C^{1+\alpha,1+\alpha/2}(\Sigma) , \\  
		f&\mapsto  \p_\nu u \Big|_{\Sigma}. 
	\end{split}
\end{align}
Now, the inverse problem is to determine $q(x,t)$ and $g$ by using the measurements either $\Lambda^{\mathrm{P}}_{q,g}$ or $	\Lambda_{q,g}$. The last main unique recovery result is stated as follows.

\begin{thm}[Simultaneous recovery for linear parabolic equations]\label{Main Thm:Simultaneous linear}
	Let $\Omega\subset \R^n$ be a bounded domain with $C^\infty$-smooth boundary $\Gamma$.	
	For  any $q_j\in C^{2+\alpha,1+\alpha/2}(\overline{Q})$ and $g_j \in C^{2+\alpha}_0(\Omega)$, 
	suppose $\Lambda_{q_j,g_j}$ and $\Lambda^{\mathrm{P}}_{q_j,g_j}$ are the full and partial DN maps of the linear parabolic equation:   
	\begin{align}\label{IBVP of simultaneous recovery linear}
		\begin{cases}
			(\p_t-\Delta  +q_j) u=0  &\text{ in } Q,\\
			u= f   &\text{ on }\Sigma,\\
			u(x,  0)=g_j(x),   &\text{ in } \Omega,
		\end{cases}
	\end{align}
for $j=1,2$, respectively. Then we have 
\begin{itemize}
	\item[(a)] (Full data) If 
	\[
	\Lambda_{q_1,g_1}(f)=\Lambda_{q_2,g_2}(f) \text{ on }\Sigma,
	\]
	for any $f\in C^{2+\alpha, 1+\alpha/2}_0(\Sigma)$, then 
$$
g_1=g_2  \text{ in } \Omega\quad \text{ and }\quad  q_1=q_2 \text{ in }Q.
$$

	\item[(b)] (Partial data) Given an open connected subset $\Omega'\subset \Omega$ satisfying $\Gamma \subset \p \Omega'$, we assume that $q_1=q_2$ in $\Omega'\times(0,T)$, where $\Omega'$ is a connected open subset of $\Omega$ such that $\p\Omega\subset\p\Omega'$.
	and
	$$\Lambda^{\mathrm{P}}_{  q_1,g_1}(f)=\Lambda^{\mathrm{P}}_{ q_2,g_2}(f) \text{ on }\mathcal{V}_-, $$
	for any $f \in  C^{2+\alpha, 1+\alpha/2}_0 (\mathcal{V}_+),$ then
	$$
	g_1=g_2  \text{ in } \Omega\quad \text{ and }\quad  q_1=q_2 \text{ in }Q.
	$$
\end{itemize}
    
\end{thm}

It is noted that when the initial data $g_1 =g_2=0$ in $\Omega$, the logarithmic stability result for two potentials of the inverse problem associated with the linear parabolic equation with partial data has been investigated in \cite{CY2018logarithmic}.

\subsection{Background and discussion}

In this paper, we are interested in the study of inverse problems for semilinear parabolic equations. A classical result of inverse boundary value problems for semilinear parabolic equations was proposed by Isakov \cite{isakov1993uniqueness_parabolic}, where a first-order linearization technique was exploited to reduce the inverse problem associated with the nonlinear equation into its counterpart associated with a linear equation. Then one can apply some existing results for the linear equations to investigate related inverse problems for the nonlinear equations. In addition, one can also consider the second-order linearization method, which has been successfully adapted in solving some related inverse problems; see \cite{AYT2017direct,CNV2019reconstruction,KN002,sun1996quasilinear,sun1997inverse} and the references cited therein.

In recent years, various inverse problems for nonlinear hyperbolic equations have been proposed and studied. Some works mentioned above are based on solution properties to inverse problems associated with the linearized equations. It turns out that in the inverse problem study associated with nonlinear hyperbolic equations, one finds that the nonlinear interactions bring more information which enables to solve some inverse problems that are still unsolved in the setting associated with linear equations. In \cite{kurylev2018inverse}, the authors investigated inverse problems for hyperbolic equations with a quadratic nonlinearity on a globally hyperbolic $4$-dimensional Lorentzian manifold. For more related works of inverse problems for nonlinear hyperbolic equations, we refer readers to \cite{lassas2017determination,lassas2018inverse,chen2019detection,de2018nonlinear,kurylev2014einstein,wang2016quadartic,LLPT2020uniqueness,LLPT2021stability,LLL2021determining} and references cited therein. 
In addition, inverse problems for semilinear elliptic equations have been attracted a lot of attentions in recent years. By utilizing high order linearization approach, it is possible to solve several inverse problems for local and nonlocal nonlinear elliptic equations, and we refer readers to \cite{LLLS2019nonlinear,FO2019semilinear,LLLS2019partial,LLST2020inverse,LL2020inverse,lai2019global,lin2020monotonicity,LZ2020partial,KU2019derivative_partial,KU2019remark,kian2020partial,carstea2020recovery,carstea2020inverse} for more detailed discussions.

The study of inverse problems on simultaneously recovering an unknown source and its surrounding inhomogeneous medium has also received considerable attentions recently in the literature due to its connection to many cutting-edge applications, including the photo- and thermo-acoustic tomography \cite{A1}, magnetic anomaly detection via the geomagnetic monitoring \cite{A3,A4} and quantum mechanics \cite{A5,A6}. Here, in the setup described in the previous section, say e.g. in \eqref{eq:parabolic simul}, the initial data $g$ and $b^{(0)}$ for $b$ in \eqref{eq:source1} represent the source terms, whereas the other terms in \eqref{eq:source1} of $b$ represent the medium effects. In \cite{LLL2021determining}, the simultaneous recovery for inverse problems associated with semilinear hyperbolic systems with unknown sources and nonlinearities was studied. In this paper, we consider the simultaneous recovery for inverse problems associated with semilinear parabolic systems. It is remarked that we develop new strategies which enable us to deal with more general source and medium configurations in the semilinear parabolic setup than the semilinear hyperbolic case. Finally, we would like to mention in passing some related physical applications that can be described by the semilinear parabolic systems in our study, including the heat diffusion \cite{ganji2018nonlinear}, mean-field game theory \cite{cardaliaguet2010notes,gueant2011mean} and phase field theory \cite{boettinger2002phase,karma2001phase}. The inverse problems proposed and studied in this paper can be connected to those practical applications.

The rest of the paper is structured as follows. In Section \ref{Sec 2}, we study the well-posedness of the initial boundary value problems for the semilinear parabolic equations under suitable assumptions. In Section \ref{Sec 3}, we establish the conditional stability estimates, and show the unique determination by utilizing either passive or active measurements. We prove Theorems \ref{Main Thm:Simultaneous} and \ref{Main Thm:Simultaneous linear} in Section \ref{Sec 4}. Finally,  for the sake of completeness, we review some basic properties including CGO solutions and weak maximum principle for linear parabolic equations.

\section{Well-posedness of the forward problems}\label{Sec 2}

This section is  devoted to studying
the local and global well-posedness for initial-boundary value problems of  semilinear parabolic equations, respectively. Let us consider  the following semilinear parabolic equation: 
\begin{align}\label{eq:parabolic simul_forward}
	\begin{cases}
		u_{t}- \nabla\cdot(\tilde\gamma\nabla u) +b(x,t,u)=0  &\text{ in }\ Q,\\
		u=\tilde f   &\text{ on }\ \Sigma,\\
		u(x,  0)=\tilde g(x)    &\text{ in }\ \Omega,
	\end{cases}
\end{align}
where $\tilde\gamma$ is  symmetric and uniformly positive definite   with $\tilde\gamma
\in C^{1+\alpha, \alpha/2}(\overline{Q}; \mathbb R^{n\times  n})$, 
$\tilde g\in C^{2+\alpha}(\overline\Omega)$,  $\tilde f\in  C^{2+\alpha, 1+\alpha/2}(\overline\Sigma)$ for $\alpha\in (0, 1)$, 
and $b$ satisfies  the following conditions:
\begin{equation}\label{e9}
	b\in  C^2(\overline{Q}\times\mathbb R)\quad\mbox{and}\quad b(\cdot,  \cdot, 0)=0 \ \mbox{in  }Q.
\end{equation}

As  a preliminary, we recall the well-posedness result
and the Schauder  estimate for linear parabolic equations, which can be found in \cite{ladyzhenskaia1988linear}.
\begin{lem}\label{wellpose1}
	For $\alpha\in (0, 1)$, assume that  $\tilde\gamma
	\in C^{1+\alpha, \alpha/2}(\overline{Q}; \mathbb R^{n\times  n})$ and $q\in C^{\alpha,\alpha/2}(\overline Q)$.
	For  any  $\tilde g\in C^{2+\alpha}(\overline\Omega)$,  $\tilde  f\in  C^{2+\alpha, 1+\alpha/2}(\overline\Sigma)$  
	and  $h\in C^{\alpha,\alpha/2}(\overline Q)$ with  the   compatibility conditions:
	\begin{align}\label{compatibility conditions}
		\tilde g(x)=\tilde f(x, 0)\ \mbox{and}\ 
		\tilde f_t(x, 0)=\nabla\cdot(\tilde \gamma(x,  0)\nabla \tilde g(x))-q(x, 0)\tilde g(x)+h(x, 0) \  \mbox{ on }\Gamma, 
	\end{align}
	the following linear parabolic equation: 
	\begin{align}
		\begin{cases}
			u_{t}- \nabla\cdot(\tilde\gamma \nabla u) +qu= h  &\text{ in }\ Q,\\
			u=\tilde f   &\text{ on }\ \Sigma,\\
			u(x,  0)=\tilde g(x)    &\text{ in }\ \Omega,
		\end{cases}
	\end{align}
	admits a unique solution $u\in C^{2+\alpha, 1+\alpha/2}(\overline{Q})$.  Moreover,  
	$$\norm{u}_{C^{2+\alpha, 1+\alpha/2}(\overline{Q})}\leq 
	C\Big(\norm{\tilde f}_{C^{2+\alpha, 1+\alpha/2}(\overline\Sigma)}
	+\norm{\tilde g}_{C^{2+\alpha}(\overline\Omega)}
	+\|h\|_{C^{\alpha, \alpha/2}(\overline Q)}\Big).$$
\end{lem}

Note that,  if $h=0$ in $Q$, $\tilde g\in C^{2+\alpha}(\overline\Omega)$  
with $\tilde g=\tilde g_{x_i}=\tilde g_{x_i x_j}=0\ (i, j=1, \cdots, n)$ on $\Gamma$  and $\tilde 
f\in C^{2+\alpha, 1+\alpha/2}(\overline\Sigma)$  with $\tilde f(x, 0)=\tilde f_t(x, 0)=0$ on $\Gamma$, it is straightforward
to verify that the compatibility condition \eqref{compatibility conditions} holds.

By Lemma \ref{wellpose1} and the fixed-point method,  we have the following local well-posedness for (\ref{eq:parabolic simul_forward}).
\begin{thm}[Local well-posedness]\label{Thm:local well-posedness}
	Assume that  $\tilde\gamma
	\in C^{1+\alpha,  \alpha/2}(\overline{Q}; \mathbb R^{n\times  n})$ and 
	$b$ satisfies the condition $(\ref{e9})$.  Then there exists a positive constant $\delta$,  
	such that for any $(\tilde f, \tilde g)\in V_\delta$, the equation \eqref{eq:parabolic simul_forward} has a unique solution $u\in C^{2+\alpha, 1+\alpha/2}(\overline{Q})$,
	where 
	\begin{eqnarray*}
		&&V_\delta=\Big\{ (\tilde  f,  \tilde g)\in C^{2+\alpha, 1+\alpha/2}(\overline{\Sigma}) \times
		C^{2+\alpha}(\overline\Omega) 
		\ \Big|\ \tilde f(x, 0)=\tilde f_t(x, 0)=0 \mbox{ on }\Gamma,\\
		&&\quad\quad\quad\quad\quad\quad\quad\quad\quad \quad \quad\quad \quad\quad\quad \quad\quad\quad\tilde g=\tilde g_{x_i}=\tilde g_{x_i x_j}=0, \ 
		i, j=1, \cdots, n  \mbox{ on }\Gamma,\\
		&&\quad\quad\quad\quad\quad\quad\quad\quad\quad\quad \quad\quad \quad\quad\quad 
		\quad\quad\quad\mbox{and }\norm{\tilde f}
		_{C^{2+\alpha, 1+\alpha/2}(\overline{\Sigma})}+\norm{\tilde g}_{C^{2+\alpha}(\overline\Omega)}\leq \delta \Big\}.
	\end{eqnarray*}
\end{thm}
\begin{proof}
	The proof can be accomplished by the fixed-point technique. First, we set
	$$
	K=\Big\{ z\in  C^{\alpha, \alpha/2}(\overline{Q}) \  \Big|\ \norm{z}_{C^{\alpha, \alpha/2}(\overline{Q})}\leq 1, 
	z(\cdot, 0)=\tilde g \mbox{ in }\Omega \mbox{ and }z=\tilde f \mbox{  on }\Sigma \Big\},
	$$
	where $(\tilde f,  \tilde g)\in V_\delta$ for a sufficiently small $\delta>0$.  It is straightfoward to show that $K$ is a nonempty convex and compact 
	subset in $L^2(Q)$. Also, we define
	\begin{equation*}
		q(x,t,s):=
		\begin{cases}
			\  \dfrac{b(x,t,s)}{s}   &\text{for }s\neq0,\\
			\      b_s(x,t,0)          &\text{for }s=0.
		\end{cases}
	\end{equation*}
	For  any $z\in K$, let us consider the following linear parabolic equation: 
	\begin{eqnarray}\label{e10}
		\begin{cases}
			u_t-\nabla\cdot(\tilde\gamma\nabla u)+q_z(x,  t)u=0   &\text{ in }\ Q,\\
			u=\tilde  f &\text{  on  }\ \Sigma, \\
			u(x,0)=\tilde g(x) &\text{ in }\ \Omega,
		\end{cases}
	\end{eqnarray}
	where $q_z(x, t)=q(x, t, z(x, t))$. 
	Define the  following map: 
	$$
	\Psi(z)=u, \quad\forall  \  z\in  K,
	$$
	where $u$ is the solution to (\ref{e10}) associated to $q_z$.
	By Lemma \ref{wellpose1},  it follows  that $u\in C^{2+\alpha, 1+\frac{\alpha}{2}}(\overline{Q})$.  Moreover,  
	$$\norm{u}_{C^{2+\alpha, 1+\alpha/2}(\overline{Q})}\leq 
	C(\tilde\gamma, b, n, \Omega, T)\Big(\norm{\tilde f}_{C^{2+\alpha, 1+\alpha/2}(\overline\Sigma)}
	+\norm{\tilde g}_{C^{2+\alpha}(\overline\Omega)}
	\Big)\leq C(\tilde\gamma, b, n, \Omega, T)\delta,$$
	where $C(\tilde\gamma, b, n, \Omega, T)$ denotes a  positive  constant, depending only on 
	$\tilde\gamma$, $b$, $n$, $\Omega$ and $T$. 
	Hence, when $\delta$ is sufficiently small, such that  $\|u\|_{C^{2+\alpha, 1+\alpha/2}(\overline{Q})}\leq 1$, 
	it  holds that   $\Psi(K)\subset  K$. By the Schauder fixed-point theorem,  it is ready to show that $\Psi$
	has a fixed point  in $K$,  which is the solution to \eqref{eq:parabolic simul_forward}.
	
	The proof is complete. 
	
\end{proof}

\begin{rmk}\label{remark 2.2}
	Regarding the local well-posedness, we make several remarks.
	\begin{itemize}
		
		\item[(a)]  The condition \eqref{e9} on $b=b(x,t,y)$ is not essential  and it is 
	 for convenience to express compatibility conditions. 
	Also,	 the admissible condition on $b(x,t,y)$ is not used in the proof of the local well-posedness, but it will be utilized in the proof of our simultaneously recovering inverse problem.

		\item[(b)] In \cite{isakov1993uniqueness_parabolic}, it is assumed that the coefficient $b=b(x,y)$, which is independent of $t$ and the condition $\p_yb(x,y)\geq 0$ for any $y\in \R$. In contrast, we provide different time-dependent nonlinearities, and utilize different techniques to study related inverse problems for semilinear parabolic equations.

		\item[(c)] In order to apply the higher order linearization method, we need the infinite differentiability of the equation with respect to the given lateral boundary data $f$, which can be shown by applying the implicit function theorem of Banach spaces. To see this, let us define the following spaces. 
		 Set
		\begin{eqnarray*}
			&&X_1=\Big\{ (f, g)\in C^{2+\alpha, 1+\alpha/2}(\overline{\Sigma}) \!\times\! C^{2+\alpha}(\overline\Omega) 
			\ \Big|\  f(x, 0)\Big|_{\Gamma}\!=\!f_t(x, 0)\Big|_{\Gamma}\!=\!0,\\
			&&\quad\quad\quad\quad\quad\quad\quad\quad\quad g\Big|_{\Gamma}=g_{x_i}\Big|_{\Gamma}=g_{x_i x_j}
			\Big|_{\Gamma}=0, i, j=1, \cdots, n \Big\},\\
			&&X_2=\Big\{  u\in C^{2+\alpha, 1+\alpha/2}(\overline{Q})  \  \Big|\   
			u(x, 0)\Big|_{\Gamma}\!=\!u_t(x, 0)\Big|_{\Gamma}\!=\!0,\\
			&&\quad\quad\quad\quad\quad u(x, 0)\Big|_{\Gamma}=u_{x_i}(x, 0)\Big|_{\Gamma}=u_{x_i x_j}(x, 0)
			\Big|_{\Gamma}=0, i, j=1, \cdots, n,\\
			&&\quad\quad\quad\quad\quad   \Big[u_{t}(x, 0)- \nabla\cdot(\gamma(x, 0)\nabla u(x, 0))\Big]\Big|_{\Gamma}=0 
			\Big\},\\
			&&\mbox{and }
			X_3=\Big\{  h\in C^{\alpha, \alpha/2}(\overline{Q})\  \Big|\  h(x, 0)\Big|_\Gamma=0 \Big\}\times X_1.
		\end{eqnarray*}
			We can consider the map 
	$\mathcal G: X_1\times X_2\to X_3$ by 
	$$\mathcal G(f,g,u)=\Big( u_t-\nabla\cdot(\gamma\nabla u)+b(x,t,u),u\Big|_{\Sigma}-f,u(x,0)-g \Big).$$
		Then $\mathcal G(0,0,0)=0$ and $\mathcal  G_{u}(0,0,0): X_2\to X_3$ is given by
		$$\mathcal G_u(0,0,0)v=\Big( v_t-\nabla\cdot(\gamma\nabla v)+b_u(\cdot,\cdot,0)v,v\Big|_{\Sigma},v(x,0)\Big).$$
		It is straightforward to show that $\mathcal G_u(0, 0, 0)$ is a linear isomorphism from $X_2$ to $X_3$ by  Lemma $\ref{wellpose1}$.
		By the implicit function theorem in Banach spaces, there exists  a positive constant  $\delta$, and a  holomorphic map 
		$S:  V_\delta\to C^{2+\alpha,  1+\alpha/2}(\overline{Q})$, such that 
		for any $(f,g)\in V_\delta$, we have $G(f,g,S(f,g))=0.$  Set $u=S(f,g)$ and this  implies  the local well-posedness of  \eqref{eq:parabolic simul_forward}.
		
		\item[(d)] Notice that the maps of boundary data to the solution is $C^\infty$-Fr\'echet differentiable, then we can also derive the corresponding DN map is also $C^\infty$-Fr\'echet differentiable.
		
	\end{itemize}
\end{rmk}

Next, for a different nonlinearity, let us consider the global well-posedness of strong solutions to the  semilinear parabolic equation:
\begin{equation}\label{e1}
	\begin{cases}
		u_t-\nabla\cdot(\gamma\nabla u)+a(x,t,u)=0   &\text{ in }\ Q,\\
		u=f &\text{ on  }\ \Sigma,\\
		u(x,0)=g(x) &\text{ in }\ \Omega,
	\end{cases}
\end{equation}
where $\gamma$ is  symmetric and  uniformly positive definite   with $\gamma\in  C^{1,  0}(\overline{Q}; \mathbb R^{n\times n})$,  $g\in H_0^1(\Omega)$,  
$f\in H^{\frac{3}{2},\frac{3}{4}}(\Sigma)$ with $f(\cdot,  0)=0$ on $\Gamma$, 
and $a: Q\times\mathbb  R\rightarrow\mathbb R$ satisfies 
\begin{equation}\label{e2}
	a(\cdot, \cdot,  0)\in L^2(Q),\quad a(x, t, \cdot)\in C^1(\mathbb R) 
\end{equation} and the  increasing condition (\ref{condition of nonlinear f at infinity data}).

The global well-posedness result of (\ref{e1}) is stated  as follows. 
\begin{thm}[Global well-posedness]	Assume that  $a$ satisfies  $(\ref{e2})$  and $(\ref{condition of nonlinear f at infinity data})$. 
	Then for any  $g\in H_0^1(\Omega)$ and
	$f\in H^{\frac{3}{2},\frac{3}{4}}(\Sigma)$ with $f(\cdot,  0)=0$ on $\Gamma$,   the semilinear parabolic equation $(\ref{e1})$ admits a unique strong solution $u\in H^{2, 1}(Q)$. 
\end{thm}	
\begin{proof}
	First,  let us set 
	\begin{equation*}
		q(x,t,s):=
		\begin{cases}
			\  \dfrac{a(x,t,s)-a(x,t,0)}{s}   &\text{for }s\neq0,\\
			\      a_s(x,t,0)          &\text{for }s=0.
		\end{cases}
	\end{equation*}
	For any $z\in L^{2}(Q)$, consider the  following linear parabolic equation:
	\begin{eqnarray}\label{e4}
		\begin{cases}
			u_t-\nabla\cdot(\gamma\nabla u)+a_z(x,t)u+a(x,t,0)=0   &\text{ in }\ Q,\\
			u=f &\text{  on  }\ \Sigma, \\
			u(x,0)=g(x) &\text{ in }\ \Omega,
		\end{cases}
	\end{eqnarray}
	where $a_z(x,t)=q(x,t,z(x,t)).$ By the condition (\ref{condition of nonlinear f at infinity data}), we have that $a_z(\cdot, \cdot)\in L^{n+2}(Q).$ 
	Indeed, there exist  positive constants,  denoted by $C$,  which 
	may be different in  one place or another, such that
	\begin{align}\label{e8}
		\begin{split}
			&\displaystyle\int_Q |a_z(x, t)|^{n+2}\, dxdt
			=\int^T_0\norm{a_z(\cdot, t)}^{n+2}_{L^{n+2}(\Omega)}\, dt\\
			\displaystyle\leq &C\!+\!C\int^T_0  e^{\norm{a_z(\cdot, t)}^2_{L^{n+2}(\Omega)}}\, dt
			\displaystyle
			=\!C\!+\!C\int^T_0\sum\limits_{j=0}^{\infty} \frac{1}{j!} \norm{a_z(\cdot, t)}_{L^{n+2}(\Omega)}^{2j} \, dt\\
			\displaystyle=&C+\!C\!\int^T_0\sum\limits_{j=0}^{n+2} \frac{1}{j!} \norm{a_z(\cdot, t)}_{L^{n+2}(\Omega)}^{2j} \, dt+
			C\!\int^T_0\sum\limits_{j=n+3}^{\infty} \frac{1}{j!} \norm{a_z(\cdot, t)}_{L^{n+2}(\Omega)}^{2j} \, dt\\
			\displaystyle\leq & C+C\int^T_0\sum\limits_{j=n+3}^{\infty} \frac{1}{j!} \norm{a_z(\cdot, t)}_{L^{n+2}(\Omega)}^{2j} \, dt\\
			\displaystyle= &
			C+C\int^T_0\sum\limits_{j=n+3}^{\infty} \frac{1}{j!} \left(\int_\Omega |a_z(x, t)|^{n+2}dx\right)^{\frac{2j}{n+2}}\, dt\\
			\displaystyle\leq &C\!+\!C\int^T_0\sum\limits_{j=n+3}^{\infty} \frac{C^j}{j!} \int_\Omega |a_z(x, t)|^{2j}\, dx dt
			\leq C\!+\!C\int_Q  e^{C|a_z(x, t)|^2}dxdt.
		\end{split}
	\end{align}
	By  the condition (\ref{condition of nonlinear f at infinity data}), for any $\epsilon>0$, there always is a positive constant $C_\epsilon$, such that 
	for any $z\in L^2(Q)$, it holds that
	$$
	|a_s(x, t, z(x, t))|^2\leq \epsilon \mbox{ln}|z(x,  t)|+C_\epsilon.
	$$
	Hence,  for a sufficient small $\epsilon$,  
	\begin{align}\label{e5}
		\begin{split}
		&	\displaystyle\int_Q  e^{C|a_z(x, t)|^2}\, dxdt\leq 
			\int_Q e^{C[\epsilon\mbox{ln}(1+|z(x, t)|)+C_\epsilon]}\, dxdt\\
			\displaystyle\leq &C\int_Q (1+|z(x, t)|)^{C\epsilon}\, dxdt 
			\leq  C\LC 1+\norm{z}^2_{L^2(Q)}\RC.
		\end{split}
	\end{align}
	(\ref{e8}) and (\ref{e5}) imply that $a_z\in L^{n+2}(Q)$. 
	
	\smallskip
	
	By \cite{ladyzhenskaia1988linear},   the linear parabolic equation (\ref{e4}) admits   a unique strong solution 
	$u\in H^{2,  1}(Q)$.    Moreover, by the energy estimate,  it is easy to find that 
	\begin{align}\label{e6}
		\begin{split}
			&\displaystyle\norm{u}^2_{L^2(0, T; H^1(\Omega))\cap C([0, T]; L^2(\Omega))}+
			\norm{u_t}^2_{L^2(0, T; H^{-1}(\Omega))}\\
			\displaystyle\leq & Ce^{T\norm{a_z}^2_{L^{n+2}(Q)}}\left(\norm{a(\cdot, \cdot, 0)}_{L^2(Q)}^2+
			\norm{g}^2_{L^2(\Omega)}+\norm{f}^2_{H^{1,  0}(\Sigma)}\right).
		\end{split}
	\end{align}
	Define the following map:
	$$\mathcal G: L^{2}(Q)\to L^{2}(Q)$$ by
	$$\mathcal G(z)=u,$$
	where $u$ is the  solution to the equation (\ref{e4}) associated to $a_z$. 
	Obviously, $\mathcal{G}$ is well-posed and compact. 
	Define 
	$$
	V=\Big\{ z\in L^2(Q)\ \Big|\  \norm{z}_{L^2(Q)}\leq C^* \Big\},
	$$
	where $C^*$ will be  specified later.  By (\ref{e8})-(\ref{e6}),  
	$$
	\norm{u}_{L^2(Q)}^2
	\leq  C\left(\norm{a(\cdot, \cdot, 0)}_{L^2(Q)}^2+
	\norm{g}^2_{L^2(\Omega)}+\norm{f}^2_{H^{1, 0}(\Sigma)}\right)\left(1+\norm{z}_{L^2(Q)}\right).
	$$
	Indeed, we may choose $\epsilon=1/C$ in  (\ref{e5}). 	
	It follows that  there exists a $C^*>0$, such that 
	$\mathcal{G}(V)\subseteq V$.  By the Schauder fixed point theorem, it is easy to check that 
	$\mathcal{G}$ has a fixed point in $V$, which is the solution to (\ref{e1}) in  $H^{2, 1}(Q)$.
	\end{proof}

\section{Unique determination of initial data}\label{Sec 3}

In this section, we present proofs of Theorems~\ref{Main Thm 1} and \ref{Main Thm 2} concerning the first two inverse problems of this paper.

\subsection{Determination by passive measurement}

In  order  to prove Theorem \ref{Main Thm 1},  
we first present two Carleman estimates for the following linear parabolic equation:
\begin{eqnarray}\label{l1}
	\begin{cases}
		u_t-\nabla\cdot(\gamma\nabla u)+A(x,t)u=F(x, t)   &\text{ in }\ Q,\\
		u=0 &\text{  on  }\ \Sigma, \\
		u(x,0)=g(x) &\text{ in }\ \Omega,
	\end{cases}
\end{eqnarray}
where $\gamma$  is the same as the one  in (\ref{eq:parabolic1}),  $A\in L^\infty(0, T; L^{2n}(\Omega))$,   $F\in  L^2(Q)$  and $g\in H^1_0(\Omega)$.

As preliminaries,  for two parameters  $\lambda, \mu\geq 1$, we introduce the following functions:
$$
\eta(x, t)=\displaystyle\frac{e^{\mu \psi(x)}-e^{2\mu\norm{\psi}_{C(\overline\Omega)}}}{t^2(T-t)^2}, 
\quad   \varphi(x, t)=\displaystyle\frac{e^{\mu\psi(x)}}{t^2(T-t)^2} \quad \mbox{ and }  \quad
\theta_1(x, t)=e^{\lambda \eta(x, t)}, 
$$
where  $\psi(\cdot)\in C^4(\overline{\Omega})$ satisfies that $\psi(x)>0$ in  $\Omega$, 
$|\nabla  \psi(x)|>0$ in $\overline{\Omega}$  and  
$$\sum\limits_{i,  j=1}^n \gamma_{i j}\psi_{x_i}\nu_j\leq 0\quad\mbox{  on   }\quad (\Gamma\setminus\Gamma_0)\times(0, T).$$
Also, for any $L>0$, there exist $t_0\in (0, T)$ and $K>0$,  such that
$$
K+t_0<\min\left\{1,  \frac{1}{2L}\right\}.
$$
Set $\theta_2(t)=\displaystyle\frac{1}{K+t_0-t}$ for  $t\in [0,  t_0]$.

The first Carleman estimate is stated as follows.

\begin{lem}\label{lemma1}
	There  exist  positive  constants $\lambda_0$, $\mu_0$ and $C$,  such that for any $\lambda\geq  \lambda_0$
	and $\mu\geq  \mu_0$,   the following estimate  holds for any solution to $(\ref{l1})$:
	\begin{align}\label{l2}
		\begin{split}
			&\displaystyle\int_Q \theta_1^2\Big(
			\lambda\mu^2\varphi|\nabla u|^2+\lambda^3\mu^4\varphi^3 u^2\Big)\, 
			dxdt\\
			\leq & C\displaystyle\int_Q \theta_1^2F^2\, dxdt+
			C\displaystyle\int^T_0\int_{\Gamma_0}
			\theta_1^2\lambda\mu \varphi\left.|\p_\nu u \right|^2\, dSdt.
		\end{split}
	\end{align}
\end{lem}

\begin{proof}
	The proof is inspired by \cite[Theorem 2.2]{yuan2017conditional}.  In fact, when $A\equiv 0$, the estimate 
	(\ref{l2}) holds true for any solution to (\ref{l1}). If $A\in L^\infty(0, T; L^{2n}(\Omega))$,
	we have that
	\begin{align*}
		\begin{split} 
			&\displaystyle\int_Q \theta_1^2\Big(
			\lambda\mu^2\varphi|\nabla u|^2+\lambda^3\mu^4\varphi^3 u^2\Big)\, 
			dxdt\\
			\leq  &C\displaystyle\int_Q \theta_1^2(F-Au)^2\, dxdt+
			C\displaystyle\int^T_0\int_{\Gamma_0} \theta_1^2 \lambda\mu
			\varphi\left|\partial_\nu u\right|^2\, dSdt.
		\end{split}
	\end{align*}
	Notice that when $n\geq 3$,
	\begin{align*}
		&\displaystyle\int_Q  \theta_1^2A^2u^2\, dxdt 
		\leq \int^T_0\norm{A}_{L^{2n}(\Omega)}^2\norm{\theta_1 u}_{L^2(\Omega)}\norm{\theta_1 u}
		_{L^{\frac{2n}{n-2}}(\Omega)}\, dt\\
		\leq  &C\int_0^T\Big( \norm{\theta_1 u}^2_{L^2(\Omega)}+
		\norm{\nabla (\theta_1 u)}^2_{L^2(\Omega)}\Big) \, dt
		\leq C\int_Q \theta_1^2\Big(|\nabla u|^2+\lambda^2\mu^2\varphi^2  u^2\Big)\, dxdt.
	\end{align*}
	When  $n=1$ and $n=2$,  the term $\norm{\theta_1 u}
		_{L^{\frac{2n}{n-2}}(\Omega)}$ can be replaced by $\norm{\theta_1 u}
		_{L^{\infty}(\Omega)}$ and $\norm{\theta_1 u}
		_{L^{4}(\Omega)}$,  respectively.
	Hence,  when  $\mu_0$  is sufficiently large, (\ref{l2})  holds for any solution to (\ref{l1}).
\end{proof}

The second Carleman estimate is  given as follows.

\begin{lem}\label{lemma2} Assume that $T\in  (0, 1)$. Then there exists   a positive constant $L_0$,  such that
	for any  $L\geq L_0$,   $t_0\in  (0, T)$  and $K>0$ with 
	$$
	K+t_0<\min\left\{1,  \displaystyle\frac{1}{2L}\right\},
	$$
	one can  always find  positive constants $\lambda_0$ and $C$,  so that for any $\lambda\geq \lambda_0$,  
	the following  estimate  holds for any solution to $(\ref{l1})$:
	\begin{align}\label{l3}
		\begin{split}
			&\displaystyle\int_0^{t_0}\int_\Omega \theta_2^{2\lambda}\Big(
			\lambda\theta_2^2u^2+L \sum\limits_{i, j=1}^n \gamma_{i j}u_{x_i}u_{x_j}\Big)\, 
			dxdt
			+\displaystyle\int_\Omega \frac{\lambda}{(K+t_0)^{2\lambda+1}}u^2(x, 0)\, dx
			\\
			\leq  &\displaystyle\int_\Omega\frac{\lambda}{K^{2\lambda+1}}u^2(x, t_0)dx
			+\displaystyle\int_\Omega\frac{1}{(K+t_0)^{2\lambda}}\sum\limits_{i, j=1}^n\gamma_{i  j}(x, 0)u_{x_i}(x, 0)u_{x_j}(x, 0)\, dx\\
			&\quad
			+C\displaystyle\int_0^{t_0}\int_\Omega \theta_2^{2\lambda} F^2 \, dxdt.
		\end{split}
	\end{align}
\end{lem}

\begin{proof}  
	The proof can be adapted from that of \cite[Theorem 2.4.1]{yu2021PhD} 
	for the Carleman  estimate of stochastic degenerate parabolic 
	equations. We sketch the necessary modifications in what follows.  
	
	First, for any $\lambda\geq 1$,  we set  $z=\theta_2^\lambda(t) u$.  Then it is straightforward to show that
	\begin{align*}
		&2\theta_2^\lambda \left[
		-\lambda\theta_2  z-\sum\limits_{i, j=1}^n (\gamma_{i  j}z_{x_i})_{x_j}\right]
		\left[
		u_t-\sum\limits_{i, j=1}^n (\gamma_{i  j}u_{x_i})_{x_j}\right]\\
		=&-(\lambda\theta_2  z^2)_t+\lambda  \theta_2^2 z^2-
		2\sum\limits_{i, j=1}^n(\gamma_{i  j}z_{x_i}z_t)_{x_j}  
		+\sum\limits_{i, j=1}^n(\gamma_{i  j}z_{x_i}z_{x_j})_{t}
		-\sum\limits_{i, j=1}^n\gamma_{i  j, t}z_{x_i}z_{x_j}\\
		&\quad
		+2\left[\lambda\theta_2  z+\sum\limits_{i, j=1}^n (\gamma_{i  j}z_{x_i})_{x_j}\right]^2.
	\end{align*}
	Integrating the above equality on $\Omega\times (0,  t_0)$,  we obtain that
	\begin{align}\label{l4}
		\begin{split}
			&\displaystyle\int_0^{t_0} \int_\Omega
			\lambda  \theta_2^2 z^2\, dxdt+
			\int_\Omega\sum\limits_{i, j=1}^n \gamma_{i j}(x, t_0)z_{x_i}(x, t_0)z_{x_j}(x, t_0)\, dx
			+\int_\Omega  \lambda\theta_2(0) z^2(x, 0)\, dx\\
			\leq  & \displaystyle\int_\Omega 
			\sum\limits_{i, j=1}^n \gamma_{i j}(x, 0)z_{x_i}(x, 0)z_{x_j}(x, 0)\, dx
			+\int_\Omega  \lambda\theta_2(t_0) z^2(x, t_0)\, dx\\
			&\quad+\displaystyle\int_0^{t_0} \int_\Omega \left|\sum\limits_{i, j=1}^n 
			\gamma_{i j, t}z_{x_i}z_{x_j}\right|\, dxdt
			+\int_0^{t_0} \int_\Omega \theta_2^{2\lambda} (Au+F)^2\, dxdt.
		\end{split}
	\end{align}
	On the other hand, it is noticed that
	$$
	\displaystyle 2\theta_2^{2\lambda} u\left[
	u_t-\sum\limits_{i, j=1}^n(\gamma_{i j}u_{x_i})_{x_j}
	\right]=(\theta_2^{2\lambda}u^2)_t
	-2\sum\limits_{i, j=1}^n(\gamma_{i j}z_{x_i}z)_{x_j}
	-2\lambda\theta_2^{2\lambda+1}u^2
	+2\sum\limits_{i, j=1}^n \gamma_{i j}u_{x_i}u_{x_j}.
	$$
	This implies that  for any $L>0$,
	\begin{align}\label{l5}
		\begin{split}
			&\displaystyle 2L\int^{t_0}_0\int_\Omega
			\sum\limits_{i, j=1}^n
			\gamma_{i j}z_{x_i}z_{x_j}\, dxdt
			+L\int_\Omega
			\theta_2^{2\lambda}(t_0)u^2(x, t_0)\, dx\\
			\leq &\displaystyle
			2L\lambda\int^{t_0}_0\int_\Omega
			\theta_2^{2\lambda+1}  u^2\, dxdt+
			L\int_\Omega \theta_2^{2\lambda}(0)u^2(x, 0)\, dx
			+2L\int^{t_0}_0\int_\Omega
			\theta_2^{2\lambda}u(Au+F)\, dxdt.
		\end{split}
	\end{align}
	By (\ref{l4}),  (\ref{l5}) and the definition of $\theta_2$,  it follows that 
	\begin{align*}
		&\int_0^{t_0}\int_\Omega
		\Big(\lambda\theta_2^{2\lambda+2}u^2+
		2L\theta_2^{2\lambda} 
		\sum\limits_{i, j=1}^{n} \gamma_{i j} u_{x_i}u_{x_j}
		\Big)\, dxdt\\
		&\quad+\int_\Omega \frac{\lambda}{(K+t_0)^{2\lambda+1}}u^2(x, 0)\, dx
		+\int_\Omega\frac{1}{K^{2\lambda}}
		\sum\limits_{i, j=1}^n \gamma_{i j}(x, t_0) u_{x_i}(x, t_0) u_{x_j}(x, t_0)\, dx\\
		\leq & L\int_\Omega \frac{1}{(K+t_0)^{2\lambda}}u^2(x, 0)\, dx
		+\int_\Omega \frac{1}{(K+t_0)^{2\lambda}}
		\sum\limits_{i, j=1}^{n} \gamma_{i j}(x, 0) u_{x_i}(x, 0)u_{x_j}(x, 0)\, dx\\
		&\quad+\int_\Omega \frac{\lambda}{K^{2\lambda+1}} u^2(x, t_0)\, dx+\int^{t_0}_0\int_\Omega\Big(
		2L\lambda\theta_2^{2\lambda+1}u^2+
		\Big|\sum\limits_{i,  j=1}^n  \gamma_{i j, t}z_{x_i}z_{x_j}\Big|\Big)\, dxdt\\
		&\quad+\int^{t_0}_0\int_\Omega
		\theta_2^{2\lambda}\Big[
		Lu^2+(L+1)(Au+F)^2\Big]\, dxdt.
	\end{align*}
	Furthermore, we notice that $\theta_2(t)\geq\frac{1}{K+t_0}>2L$. Also, for any $\epsilon>0$,  
	\begin{align*}
		&\int_0^{t_0} \int_\Omega \theta_2^{2\lambda} A^2u^2\, dxdt \\
		\leq & \int^{t_0}_0  \theta_2^{2\lambda} \norm{A}^2_{L^{2n}(\Omega)}
		\norm{u}_{L^2(\Omega)}\norm{u}_{L^{\frac{2n}{n-2}}(\Omega)}\, dt\\
		\leq &\epsilon \int_0^{t_0} \int_\Omega \theta_2^{2\lambda}|\nabla  u|^2 \, dxdt+
		C\int_0^{t_0} \int_\Omega\theta_2^{2\lambda} u^2 \, dxdt.
	\end{align*}
	Hence, 
	for sufficiently large $L$ and $\lambda$, 
	it follows that
	\begin{align*}
		&\int_0^{t_0}\int_\Omega
		\Big(\lambda\theta_2^{2\lambda+2}u^2+
		2L\theta_2^{2\lambda} 
		\sum\limits_{i, j=1}^{n} \gamma_{i j} u_{x_i}u_{x_j}
		\Big)\, dxdt+\int_\Omega \frac{\lambda}{(K+t_0)^{2\lambda+1}}u^2(x, 0)\, dx\\
		\leq &\int_\Omega \frac{1}{(K+t_0)^{2\lambda}}
		\sum\limits_{i, j=1}^{n} \gamma_{i j}(x, 0) u_{x_i}(x, 0)u_{x_j}(x, 0)\, dx\\
		&\quad+\int_\Omega \frac{\lambda}{K^{2\lambda+1}} u^2(x, t_0)\, dx
		+C\int^{t_0}_0\int_\Omega
		\theta_2^{2\lambda}F^2\, dxdt.
	\end{align*}
	This implies the desired estimate (\ref{l3}).
	
	The proof is complete. 
\end{proof}

Based  on Lemmas \ref{lemma1} and \ref{lemma2},  
one has the following conditional stability result for the inverse source problem  of (\ref{l1}). 
\begin{lem}\label{lemma3} 
	For any $M>0$,  if 
	\[
	\norm{g}_{H^1_0(\Omega)}+\norm{F}_{L^2(Q)}\leq M,  
	\]
	there exist  positive constants $C$  and $\delta_0\in (0, 1)$, 
	depending only on $n,  T$  and $\Omega$, such that 
	the following  estimate holds for any solution to $(\ref{l1})$:
	\begin{equation}\label{l6}
		\displaystyle\norm{u(\cdot, 0)}_{L^2(\Omega)}^2
		\leq 
		\frac{C(M+1)}{\delta_0}  
		\norm{(F, \p_\nu u)}-
		\frac{CM^2}{\ln[\delta_0  \norm{(F, \p_\nu u)}]},
	\end{equation}
	where  $\norm{(F, \p_\nu u)}=\left(\norm{F}_{L^2(Q)}^2+
	\norm{\partial_\nu u}^2_{L^2(\Gamma_0\times(0, T))}\right)^{1/2}.$

\end{lem}

\begin{proof}  

	Without loss of generality,  we assume that $T<1$.  For any $t_1\in 
	(0, T)\cap(0, \frac{2}{3})\cap(0, \frac{1}{3L})$ with $L$ being the constant  in 
	Lemma \ref{lemma2},  
	choose $K=\frac{t_1}{2}$ and $t_0\in [\frac{t_1}{2}, t_1]$.
	Then,
	$$
	\displaystyle K+t_0\leq \frac{3}{2}t_1<\min\left\{1, \frac{1}{2L}\right\}
	$$
	and
	$$
	\displaystyle \LC \frac{t_1+2t_0}{2}\RC^{-2\lambda}
	=(K+t_0)^{-2\lambda} \leq \theta_2^{2\lambda}(t)
	\leq \Big(\frac{2}{t_1}\Big)^{2\lambda}, \quad\mbox{for any }\lambda\geq \lambda_0
	\mbox{  and }t\in  [0, t_0].
	$$
	
	By Lemma \ref{lemma2}, 
	\begin{align*}
		\begin{split}
			&\displaystyle\lambda\int_\Omega
			\LC \frac{t_1+2t_0}{2}\RC ^{-2\lambda-1}u^2(x, 0) \, dx\\
			\leq & \displaystyle C\int_\Omega 
			\LC \frac{t_1+2t_0}{2}\RC^{-2\lambda}|\nabla  u(x, 0)|^2\, \, dx+C\lambda\Big(\frac{2}{t_1}\Big)^{2\lambda+1}
			\left[\int_\Omega  u^2(x, t_0)dx+\int_Q F^2(x, t)\, dxdt\right].
		\end{split}
	\end{align*}
	This implies that
	\begin{align}\label{l7}
		\begin{split}
			&\displaystyle\int_\Omega  u^2(x, 0)\, dx\\
			\leq  &\displaystyle\frac{C}{\lambda}\int_\Omega  |\nabla u(x, 0)|^2\, dx\\
			&\quad+C\LC \frac{t_1+2t_0}{2}\RC^{2\lambda}
			\Big(\frac{2}{t_1}\Big)^{2\lambda+1}
			\left[\int_\Omega  u^2(x, t_0)\, dx+\int_Q F^2(x, t)\, dxdt\right]\\
			\leq &\displaystyle\frac{C}{\lambda}\int_\Omega  |\nabla u(x, 0)|^2\, dx
			+C9^\lambda \norm{(F,  t_0)}^2,
		\end{split}
	\end{align}
	where $\norm{(F,  t_0)}^2:=\displaystyle\int_\Omega  u^2(x, t_0)\, dx+\int_Q F^2(x, t)\, dxdt$.

	On the other  hand,  by  Lemma \ref{lemma1}, for fixed parameters $\lambda$ and $\mu$, it holds that
	$$
	\displaystyle\int^{t_0}_{\frac{t_0}{2}}
	\int_\Omega \LC  u^2+|\nabla u|^2\RC \, dxdt
	\leq C\int_Q F^2dxdt+C\int^T_0\int_{\Gamma_0} \left|\partial_\nu u\right|^2\, dSdt.
	$$ 
	Hence, there exists a $\hat  t\in (\frac{t_0}{2}, t_0)$,  such that
	$$
	\displaystyle
	\int_\Omega \LC  u^2(x, \hat t)+|\nabla u(x, \hat t)|^2 \RC \, dx
	\leq C\int_Q F^2dxdt+C\int^T_0\int_{\Gamma_0} \left|\partial_\nu u\right|^2\, dSdt.
	$$
	By the standard energy estimate,
	\begin{align}\label{l8}
		\begin{split}
			&\displaystyle\int_\Omega u^2(x, t_0) \, dx \\
			\leq  &	C\int_\Omega u^2(x,  \hat t)\, dx+C\int^{t_0}_{\hat  t}\int_\Omega(u^2+ F^2)\, dxdt\\
			\displaystyle\leq & C \int_Q F^2\, dxdt+C\int^T_0\int_{\Gamma_0} \left|\partial_\nu u\right|^2\, dSdt.
		\end{split}
	\end{align}

	By (\ref{l7}) and (\ref{l8}), it holds that
	\begin{equation}\label{l9}
		\int_\Omega u^2(x, 0)\, dx\leq \frac{C}{\lambda}\int_\Omega  |\nabla u(x, 0)|^2\, dx+C9^\lambda
		\norm{(F, u_\nu)}^2.
	\end{equation}
	Take  
	$$
	\delta_0\in (0, e^{-\lambda_0\ln 9})\quad\mbox{ and }
	\quad \lambda=\frac{1}{\ln 9}\ln \LC \frac{\norm{(F, u_\nu)}+1}{\delta_0  \norm{(F, u_\nu)}}\RC,
	$$
	where $\lambda_0$ is  the  constant in Lemma \ref{lemma2}. Then, $\lambda\geq \lambda_0$.
	Set 
	$$\hat{u}=\displaystyle\frac{\delta_0}{\norm{(F, u_\nu)}+1}u\quad
	\mbox{  and  }
	\quad\hat  F=\displaystyle\frac{\delta_0}{\norm{(F, u_\nu)}+1}F.
	$$
	Hence, by  (\ref{l9}), 
	\begin{align*}
		&\int_\Omega \hat u^2(x, 0)\, dx\\
		\leq &	\frac{C}{\lambda}\int_\Omega |\nabla  \hat u(x,  0)|^2\, dx+
		C9^\lambda \frac{\delta_0^2}{(\norm{(F, u_\nu)}+1)^2} \norm{(F, u_\nu)}^2\\
		\leq & \frac{C}{\ln  \LC \frac{\norm{(F, u_\nu)}+1}{\delta_0 \norm{(F, u_\nu)}}\RC}
		\int_\Omega |\nabla \hat u(x, 0)|^2\, dx
		+C\frac{\delta_0 \norm{(F, u_\nu)}}{\norm{(F, u_\nu)}+1}.
	\end{align*}
	It follows that
	\begin{align*}
		\int_\Omega u^2(x, 0)\, dx\leq 
		\frac{C}{\ln  \LC \frac{\norm{(F, u_\nu)}+1}{\delta_0 \norm{(F, u_\nu)}}\RC}
		\int_\Omega |\nabla  u(x, 0)|^2\, dx
		+C\frac{(\norm{(F, u_\nu)}+1)\norm{(F, u_\nu)}}{\delta_0}.
	\end{align*}
	For any $M>0$,  
	$$
	\norm{(F, u_\nu)}\leq  CM.
	$$
	Hence, 
	\begin{align*}
		\int_\Omega u^2(x, 0)\,  dx\leq 
		\frac{CM^2}{\ln \LC \frac{\norm{(F, u_\nu)}+1}{\delta_0 \norm{(F, u_\nu)}}\RC}
		+C\frac{(M+1)\norm{(F, u_\nu)}}{\delta_0}.
	\end{align*}
	This implies the desired estimate (\ref{l6}).
\end{proof}

Now,   we come back to the proof of Theorem \ref{Main Thm 1}.

\begin{proof}[Proof of Theorem $\ref{Main Thm 1}$]
	For any $a\in \mathcal{A}_T$ and two initial values $g_1, g_2\in H^1_0(\Omega)$,  
	let $\wt u=u_1 -u_2$, 
	where $u_j$   $(j=1, 2)$ are the solutions to \eqref{IBVP for thm 1 for j=1,2}
	associated to $g_j$. Then $\wt u\in H^{2, 1}(Q)$ is the solution to the 
	following parabolic equation: 
	\begin{align}\label{equ proof of thm 1 for j=1,2}
		\begin{cases}
			\wt u_{ t}-\nabla\cdot(\gamma\nabla \wt u)+A(x,t)\wt u=0  &\text{ in }\ Q,\\
			\wt u=0  &\text{ on }\ \Sigma,\\
			\wt u(x, 0)=g_1-g_2, 
			&\text{ in }\ \Omega,
		\end{cases}
	\end{align}
	with
	\begin{align*}
		A(x,t)\wt u= a(x,t,u_1)-a(x,t,u_2) = \LC \int_0^1 a_u (x,t,su_1 + (1-s)u_2)\, ds \RC\cdot \wt u.
	\end{align*}
	with $A(x,t)= \displaystyle\int_0^1 a_u (x,t,su_1 + (1-s)u_2)\, ds$. Similar to \cite[Theorem 3.2]{LLL2021determining}, 
	we can prove that $A\in L^\infty(0, T; L^{2n}(\Omega))$. By Lemma \ref{lemma3}, 
	for any $M>0$,  if 
	$\norm{g_1-g_2}_{H^1_0(\Omega)}\leq M, $
	there exist  positive constants $C$  and $\delta_0\in (0, 1)$, 
	depending only on $n,  T$  and $\Omega$, such that 
	$$
	\displaystyle\norm{\wt u(\cdot, 0)}_{L^2(\Omega)}^2
	\leq 
	\frac{C(M+1)}{\delta_0}  
	\norm{\p_\nu  \wt u }_{L^2(\Gamma_0\times(0, T))}-
	\frac{CM^2}{\ln\LC \delta_0  \norm{\p_\nu \wt u}
		_{L^2(\Gamma_0\times(0, T))}\RC}.
	$$
	This proves the desired estimate (\ref{Stability estimate in Thm 1}).
\end{proof}

Furthermore,  there is a counterexample showing that if  $a$ is unknown,  the passive measurement cannot uniquely determine all unknowns. 
\begin{thm}[Non-uniqueness]\label{thm:2} 
	Suppose that $\gamma=\LC \gamma^{ij}(x)\RC_{i,j=1}^n \in 
C^{2, 1}(\overline{Q}; \R^{n\times n})$ is  symmetric uniformly positive definite, $a_j\in\mathcal A_T$  and
$g_j\in H^1_0(\Omega)$ for $j=1, 2$.  Denote by
 $\Lambda_{a_j,  g_j}^0$  the DN map of the following semilinear parabolic equation:  
	\begin{align}\label{llll}
	\begin{cases}
	u_{j,  t}-\nabla\cdot(\gamma\nabla u_j)+a_j(x,t,u_j)=0  &\text{ in }\ Q,\\
	u_j=0   &\text{ on }\ \Sigma,\\
	u_j(x,0)=g_j(x),    &\text{ in }\ \Omega.
		\end{cases}
	\end{align}
Then there exist two groups of unknown sources $(g_1, a_1),  (g_2,  a_2)
\in H^1_0(\Omega)\times\mathcal A_T$,  such that 
\[
(g_1, a_1) \neq (g_2, a_2),
\]
but
\[
\Lambda^0_{g_1, a_1}=\Lambda^0_{g_2,a_2}.
\]
\end{thm}

\begin{proof} Assume that two   functions $u_1, u_2\in C^\infty(\overline{Q})$ satisfy   that  
\begin{eqnarray*}
&&u_1(\cdot, 0)\neq u_2(\cdot, 0) \mbox{ in a measurable set of }\Omega\mbox{  with positive measure}, \\[2mm]
&&\mbox{and}\  u_1(x, t)=u_2(x, t)=0 \mbox{  in  } \Omega_\epsilon\times[0, T],
\end{eqnarray*}
where $\Omega_\epsilon=\Big\{ x\in \Omega\  \Big|\      \dist(x,  \Gamma)<\epsilon  \Big\}$.
Set 
$$
A_j(x, t)=-u_{j, t}(x,t)+\nabla\cdot(\gamma\nabla u_j(x, t)), \quad \text{ for }j=1,2 \mbox{ and }(x, t)\in Q.
$$
It is easy to show  that  $u_j$   $(j=1, 2)$  are solutions to (\ref{llll}) associated  to 
\begin{align*}
&g_j(x)=u_j(x,  0)\quad\mbox{  and  } \quad a_j(x, t, u_j)=A_j(x,  t).
\end{align*}
Then,  
\[
(g_1, a_1)\neq (g_2, a_2), 
\]   
but $$\partial_\nu u_{1}\Big|_{\Gamma_0\times(0, T)}=\Lambda^0_{g_1, a_1}=
\Lambda^0_{g_2,a_2}=\partial_\nu u_{2}\Big|_{\Gamma_0\times(0, T)}=0.$$

\end{proof}

\subsection{Determination by active measurements}

This subsection is devoted to proving Theorem \ref{Main Thm 2} under the condition that 
$a\in \mathcal{B}_{T}$  (see (\ref{set C})).

\begin{proof}[Proof of Theorem $\ref{Main Thm 2}$]
	For any  $a_j\in \mathcal{B}_{T}$ $(j=1, 2)$,  
	 $$a_j(x, t, y)=a_0(x, t, y)\chi_{[0, T-\epsilon]}(t)+c_j(x,t,y)\chi_{[T-\epsilon, T]}(t), 
	$$
	where $\epsilon>0$,    $a_0\in \mathcal{A}_T$  and 
	$c_1, c_2\in \mathcal{A}_T$   with $c_1(x, t, 0)=c_2(x, t, 0)=0$ in $Q$.

	By the assumptions in Theorem $\ref{Main Thm 2}$, 
	for any  
	$g_j\in H_0^1(\Omega)$ $(j=1,  2)$,  it holds  that 
	\begin{equation}\label{eq:p1}
		\Lambda_{g_1, a_1}(f)=\Lambda_{g_2, a_2}(f),\quad \text{ for any }  f\in 
		\mathcal  E
		\mbox{ with }\supp  f\subset \Gamma_0\times[0, T]. 
	\end{equation}
	Let $u_1$ and $u_2$ be, respectively, the solutions to (\ref{IBVP for thm 2 for j=1,2}) associated to  the above 
	mentioned $(g_1,a_1, f)$ and $(g_2,a_2,f)$. 
	Set 
	\[
	\wt u=u_1-u_2.
	\]
	It is straightforward to see that
	\begin{align}\label{eq:wave2}
		\begin{cases}
			\wt u_{t}-\nabla\cdot(\gamma\nabla \wt u)+a_1(x,t,u_1(x, t))- a_2(x, t,  u_2(x, t))=0  &\text{ in }\ Q,\\[2mm]
			\wt u=0   &\text{ on }\ \Sigma,\\
			\partial_\nu \wt u=0 &\text{  on  }\ \Gamma_0\times(0, T),\\
			\wt u(x, 0)=\wt g(x)  &\text{ in }\ \Omega,
		\end{cases}
	\end{align}
	where
	$\wt g=g_1-g_2$.

	Notice that
	\begin{align*}
		&a_1(x,t,u_1(x, t))-a_2(x, t,  u_2(x, t))\\[2mm]
		=&\Big[a_1(x,t,u_1(x, t))-a_1(x, t, u_2(x, t))\Big]
		+\Big[a_1(x,t,u_2(x, t))-a_2(x, t,  u_2(x, t))\Big] \\
		=&\LC \displaystyle\int^1_0 a_{1,  u}(x, t,  su_1(x, t)+(1-s)u_2(x,  t))\, ds\RC\cdot \wt u(x, t) \\[1mm]
		&\quad+ a_1(x,t, u_2(x, t))- a_2(x, t, u_2(x, t)).
	\end{align*}
	By $a_1\in \mathcal{A}_{T}$,   
	$\displaystyle\int^1_0 f_{1,  u}(x, t,  su_1(x, t)+(1-s)u_2(x,  t))\, ds\in L^\infty(0, T; L^{2n}(\Omega))$.  
	By  Lemma \ref{lemma3},  one has that 
	\begin{align}\label{l13}
		\begin{split}
			\displaystyle\norm{g_1-g_2}^2_{L^2(\Omega)}
			\displaystyle\leq &C(u_1, u_2)\norm{a_1(\cdot, \cdot, u_2(\cdot, \cdot))
				-a_2(\cdot, \cdot, u_2(\cdot, \cdot))}_{L^2(Q)}\\
			&\quad\displaystyle-\frac{C(u_1, u_2)}{\ln[\delta_0
				\norm{a_1(\cdot, \cdot, u_2(\cdot, \cdot))-a_2(\cdot, \cdot, u_2(\cdot, \cdot))}_{L^2(Q)}]},
		\end{split}
	\end{align}
	where $C(u_1, u_2)$ denotes a constant depending on $u_1$ and $u_2$.

	Recall  the controllability result for the following semilinear parabolic equation (see \cite{dzz}):
	\begin{align}\label{l14}
		\begin{cases}
			u_{2, t}-\nabla\cdot(\gamma\nabla u_2)+a_2(x,t,u_2)=0  &\text{ in }\ \Omega\times(0, T-\epsilon),\\
			u_2=f  &\text{ on }\ \Gamma\times(0, T-\epsilon),\\
			u_2(x, 0)=g_2(x) ,  &\text{ in }\ \Omega.
		\end{cases}
	\end{align} 
	For any $g_2\in L^2(\Omega)$, 
	there  always exists  an $f^*\in L^2(\Gamma\times(0, T-\epsilon))$ with $\supp f^*\subset 
	\Gamma_0\times(0, T-\epsilon]$,  
	such that  the corresponding solution   $u_2$ to (\ref{l14})    satisfies that
	$u_2(x, T-\epsilon)=0$ in $\Omega$.
	Choose 
	$$f(x, t)=
	\begin{cases}
		\displaystyle f^*(x, t) &\mbox{ on }\ \Gamma_0\times[0, T-\epsilon),\\
		\displaystyle  0 &\mbox{ on }\ \Gamma_0\times[T-\epsilon, T].
	\end{cases}$$
	By the fact  that $c_1(\cdot, \cdot, 0)=c_2(\cdot, \cdot, 0)\equiv 0$ in $Q$,  the solution 
	$u_2=u_2(\cdot, \cdot; f)$ to (\ref{l14}) associated to  the above
	$f$  satisfies  that 
	$$
	u_2(x, t)\equiv 0\quad\mbox{ in  }\ \Omega\times [T-\epsilon, T].
	$$
	Combining the above result with (\ref{l13}),  one readily obtains 
	\[
	g_1=g_2 \quad \mbox{ in }\ \Omega,
	\]
	which proves the assertion in Theorem  $\ref{Main Thm 2}$.
\end{proof}

\section{Simultaneous recovery results for inverse problems}\label{Sec 4}

In this section, we present the proofs of Theorems~\ref{Main Thm:Simultaneous} and \ref{Main Thm:Simultaneous linear} on the simultaneous recovery results for the inverse problems. We first derive the unique determination of the coefficient for the linear parabolic equation.
To that end, let us prove some useful properties, which will be needed in the proofs of Theorems \ref{Main Thm:Simultaneous} and \ref{Main Thm:Simultaneous linear}.

\subsection{Approximation and denseness properties}

Let us begin with the Runge approximation properties for linear parabolic equations. The following approximation property will be used in the proof of Theorems \ref{Main Thm:Simultaneous} and \ref{Main Thm:Simultaneous linear} with full data.

\begin{lem}[Runge approximation with full data]\label{appro}
   Let $q\in C^{2+\alpha,1+\alpha/2}(\overline{Q})$. Then for any solutions $v_{\pm}\in L^2(0,T;H^1(\Omega))\cap H^1(0,T;H^{-1}(\Omega))$	to 
\begin{equation}
	\begin{cases}\label{to App1}
		\p_tv_+-\Delta v_+ +qv_+=0  &\text{ in }\ Q,\\
	    v_+(x,0)=0               &\text{ in }\ \Omega,
	\end{cases}
\end{equation} and
\begin{equation}\label{to App2}
	\begin{cases}
		-\p_tv_--\Delta v_- +qv_-=0  &\text{ in }\ Q,\\

		v_-(x,T)=0               &\text{ in }\ \Omega,
	\end{cases}
\end{equation}
and  any $\eta>0$, there exist solutions $V_{\pm}\in C^{2+\alpha,1+\alpha/2}(\overline{Q})$ to 
\begin{equation}\label{app1}
	\begin{cases}
		 \p_tV_+-\Delta V_++ qV_+=0   &\text{ in }\ Q,\\
	
		V_+(x,0)=0            &\text{ in }\ \Omega,
	\end{cases}
\end{equation} and
\begin{equation}\label{app2}
	\begin{cases}
		-\p_tV_--\Delta V_-+ qV_-=0   &\text{ in }\ Q,\\

		V_-(x,T)=0            &\text{ in }\ \Omega,
	\end{cases}
\end{equation}
	such that
	$$\norm{V_{\pm}-v_{\pm}}_{L^2(Q)}<\eta.$$
\end{lem}
\begin{proof}
	We only prove the case for the forward parabolic equation, and the backward one can be proved similarly. Define
	$$X=\Big\{V\in C^{2+\alpha,1+\alpha/2}(\overline{Q})\,\Big|\, 
	V \text{ is a solution to } \eqref{app1} \Big\}$$and
	$$Y=\Big\{v\in L^2(0,T;H^1(\Omega))\cap H^1(0,T;H^{-1}(\Omega)) \,\Big|\, v \text{ is a solution to } \eqref{to App1} \Big\}.$$
	We aim to show that $X$ is dense in $Y$. By the Hahn-Banach theorem, it suffices to prove the following statement: If $f\in L^2(Q)$ satisfies
	$$\int_{Q}fV\,dxdt=0, \quad \text{ for any } V\in X,$$
	then
	$$\int_{Q}fv\,dxdt=0, \quad \text{ for any }v\in Y.$$
	To this end, 
	let $f\in L^2(Q)$ and suppose $\int_{Q}fV\,dxdt=0$,  for any  $V\in X$. Consider 
	\begin{equation}
		\begin{cases}
			-\p_t\overline{V}-\Delta \overline{V} +q\overline{V}=f  &\text{ in }\ Q,\\
			\overline{V}=0 &\text{ on }\ \Sigma,\\
			\overline{V}(x,T)=0            &\text{ in }\ \Omega
		\end{cases}
	\end{equation}
and its  solution is in $H^{2, 1}(Q)$. For any $V\in X$, one has 
\begin{align*}
	0=&\int_{Q}fV\,dxdt=\int_Q (-\p_t\overline{V}-\Delta \overline{V} +q\overline{V}) V\,dxdt=\int_{\Sigma} \p_\nu \overline{V}V\, dSdt.
\end{align*}
Since $V|_{\Sigma}$ can be arbitrary function, which  is compactly supported on $\Sigma$, we must have $\p_\nu \overline{V}=0$ on $\Sigma$.
Thus, for any $ v\in Y$,
\begin{align*}
	\int_{Q}fv \,dxdt&=\int_Q (-\p_t\overline{V}-\Delta \overline{V} +q\overline{V}) v\,dxdt=\int_{\Sigma}\p_\nu\overline{V}v\,dSdt=  0,
\end{align*}
which verifies the assertion.
\end{proof}
Let $\Omega\subset \R^n$ be a connected domain, and $\Omega'$ be a connected open subset of $\Omega$ such that $\p\Omega\subset\p\Omega'$. Define $Q'=(\Omega\backslash\Omega')\times(0,T)$. Meanwhile, for given $\varepsilon>0$ and $\omega\in\mathbb{S}^{n-1}$, we set 
\begin{align*}
	&\Gamma_{+,\omega,\varepsilon}:=\Big\{x\in\Gamma\  \Big| \  \nu(x)\cdot\omega>\varepsilon\Big\}, \\
	&\Gamma_{-,\omega,\varepsilon}:=\Big\{x\in\Gamma \   \Big| \   -\nu(x)\cdot\omega>\varepsilon\Big\},\\
	&\mbox{and }\Sigma_{\pm,\omega,\varepsilon}:=\Gamma_{\pm,\omega, \varepsilon}\times(0,T).
\end{align*}
The following approximation property will be used to prove Theorems \ref{Main Thm:Simultaneous} and \ref{Main Thm:Simultaneous linear} with partial data.

\begin{lem}[Runge approximation with partial data]\label{appro2}
	Let $q\in C^{2+\alpha,1+\alpha/2}(\overline{Q})$. Then for any solutions $W_{\pm}\in L^2(0,T;H^1(\Omega))\cap H^1(0,T;H^{-1}(\Omega))$	to 
	\begin{equation}\label{to app3}
		\begin{cases}
				\p_tW_+-\Delta W_+ +qW_+=0  &\text{ in }\ Q,\\
				W_+(x,0)=0              &\text{ in }\ \Omega
		\end{cases}		
	\end{equation}and
	\begin{equation}\label{to App4}
			\begin{cases}
			\p_tW_--\Delta W_- +qW_-=0  &\text{ in }\ Q,\\
			W_-(x,T)=0              &\text{ in }\ \Omega,
		\end{cases}	
	\end{equation}
	and any $\eta>0$, there exist solutions $v_{\pm}\in C^{2+\alpha,1+\alpha/2}(\overline{Q})$ to 
	\begin{equation}\label{app3}
		\begin{cases}
			\p_tv_+-\Delta v_++ qv_+=0   &\text{ in }\ Q,\\
			v_+=0                    &\text{ on }\ \Gamma_{-,\omega,\varepsilon}\times(0,T),\\
			v_+(x,0)=0            &\text{ in }\ \Omega,
		\end{cases}
	\end{equation}and
	\begin{equation}\label{app4}
		\begin{cases}
			-\p_tv_--\Delta v_-+ qv_-=0   &\text{ in }\ Q,\\
			v_-=0                    &\text{ on }\ \Gamma_{+,\omega,\varepsilon}\times(0,T),\\
			v_-(x,T)=0            &\text{ in }\ \Omega,
		\end{cases}
	\end{equation}
	such that
	$$\|W_{\pm}-v_{\pm}\|_{L^2(Q')}<\eta.$$
\end{lem}
\begin{proof}
		We may only prove the case for forward parabolic equations. Define
	
	$$X'=\Big\{v\in C^{2+\alpha,1+\alpha/2}(\overline{Q}) \,\Big|\, v \text{ is a solution to } 
	\eqref{app3} \Big\}$$and
	$$Y'=\Big\{W\in L^2(0,T;H^1(\Omega))\cap H^1(0,T;H^{-1}(\Omega))\,\Big|\, V 
	\text{ is a solution to } \eqref{to app3} \Big\}.$$
	We aim to show that $X'$ is dense in $Z$. By the Hahn-Banach theorem again, it suffices to claim that if $f\in L^2(Q')$ satisfies
	$$\int_{Q'}fv\,dxdt=0,  \text{ for any } v\in X',$$
	then
	$$\int_{Q'}fW\,dxdt=0,  \text{ for any } W\in Y'.$$
	Let $f\in L^2(Q')$ satisfy  that $\int_{Q'}fv\,dxdt=0, \  \forall v\in X'$.  We extend $f$ to $Q$ by letting $f=0$ outside $Q'$.

	Consider 
	\begin{equation}
		\begin{cases}
			-\p_t\overline{v}-\Delta \overline{v} +q\overline{v}=f  &\text{ in }\ Q,\\
			\overline{v}=0                    &\text{ on }\ \Sigma,\\
			\overline{v}(x,T)=0            &\text{ in }\ \Omega,
		\end{cases}
	\end{equation}
	and  its solution is in $H^{2, 1}(Q).$  Then for any $v\in X'$,
	\begin{align*}
		&0=\int_{Q}fv\,dxdt=\int_Q (-\p_t\overline{v}-\Delta \overline{v} +q\overline{v}) v\,dxdt=\int_{\Sigma} \p_\nu \overline{v}v\,dSdt.
	\end{align*}
	Since $v|_{\Sigma}$ can be arbitrary function, which is compactly supported on $\Sigma\backslash(\Gamma_{-,\omega,\varepsilon}\times(0,T))$ and $v=0$ on $\Gamma_{-,\omega,\varepsilon}\times(0,T) $, we have that $\p_\nu \overline{v}=0$ on $\Sigma\backslash(\Gamma_{-,\omega,\varepsilon}\times(0,T))$.

	Next, fix a set $\Omega_1$  with
	 nonempty interior,  such that $(\Omega_1 \cap \p\Omega) \subset (\Gamma\backslash\Gamma_{-,\omega,\varepsilon})$. Then $\overline{v}=0$ on $\Omega_1\times(0,T)$. Notice that
	$$ -\p_t\overline{v}-\Delta \overline{v} +q\overline{v}=0  \text{ in } (\Omega'\cup\Omega_1)\times(0,T).$$
	Since $\Omega'\cup\Omega_1$ is open and connected, by the unique continuation principle for linear parabolic equations (for instance, see \cite{saut1987unique}), we have $\overline{v}=0$ on $\Omega'\times(0,T).$
	Hence, $\overline{v}\Big|_{\p\Omega'\times(0,T)}=\p_\nu \overline{v}\Big|_{\p\Omega'\times(0,T)}=0,$ and it follows that $$\overline{v}\Big|_{\p(\Omega\backslash\Omega')\times(0,T)}=\p_\nu \overline{v}\Big|_{\p(\Omega\backslash\Omega')\times(0,T)}=0.$$ 
	Hence, for any $W\in Y'$,
	\begin{align*}
		\int_{Q'}fW \,dxdt=\int_{Q'} (-\p_t\overline{v}-\Delta \overline{v} +q\overline{v}) W\,dxdt=\int_{\p(\Omega\backslash\Omega')\times(0,T)}\p_\nu\overline{v}W\,dSdt=  0.
	\end{align*}
	This completes the proof.
\end{proof}

\begin{rmk}
	Let us refer readers to some related approximation property for some different diffusion equations, such as \cite[Lemma 5.3]{CK2018determination}. Since the proofs of the global uniqueness results with either full data or partial data are similar, we focus on presenting the arguments for the full data case and remark the necessary modifications for the partial data case, and vice versa.
\end{rmk}

\begin{lem}[Denseness property]\label{Lem: denseness}
	Let $q_1,q_2 \in L^\infty(Q)$.
	Assume that $f\in L^\infty(Q)$,  such that 
	$$\int_{Q} fv_1v_2\, dxdt =0,$$ 
	for any $v_1$ and $v_2$, which satisfy $v_1v_2\in L^1(Q)$, and are,  respectively,  solutions to 
	\begin{align}\label{equ forward CGO}
		\begin{cases}
			\p_{ t} v_1-\Delta   v_1+q_1  v_1=0   &\text{ in }\  Q,\\
			v_1(x,0)=0     &\ \text{ in }\Omega,
		\end{cases}                     
	\end{align}
	and
	\begin{align}\label{equ backward CGO}
		\begin{cases}
			\p_{ t} v_2+\Delta v_2-q_2 v_2=0   &\text{ in }\ Q,\\
			v_2(x,T)=0      &\text{ in }\ \Omega.
		\end{cases}                       
	\end{align}
	Then $f=0$.
	In other words, the linear span of products of solutions to forward and backward 
	parabolic equations  are dense in $L^1(Q)$.
\end{lem}

\begin{proof}
	Since $q_j\in L^\infty(Q)$ for $j=1,2$, without loss of generality, we may 
	assume that  there exists a positive number $m$,  such that $q_1,  q_2\in \Big\{ q\in L^{\infty}(Q)\ \Big|\ 
	\norm{q}_{L^{\infty}(Q)}<m \Big\}$.
	First, let us  fix $\omega\in\mathbb{S}^{n-1}$. Consider $\rho>0$ to be sufficiently  large, and $(\xi,\tau)\in M:
	=\Big\{(\xi, \tau)\in\mathbb{R}^{n+1} \big| \, \xi\cdot\omega=0 \Big\}$ with $|(\xi,\tau)|^2<\rho-1$.
	Then by Proposition \ref{CGO_F}, there is a solution $v_{1,\rho}(\cdot, \cdot; \xi, \tau)$ to \eqref{equ forward CGO} such that $$v_{1,\rho}=\psi_{-,  \rho}(\theta_{+, \rho}+z_{+, \rho, q_1})$$ with $\norm{z_{+,  \rho, q_1}}_{L^2(Q)}\to 0$ as $\rho\to\infty$.
	Similarly, there is a solution $v_{2,\rho}(\cdot, \cdot)$ to the backward parabolic equation \eqref{equ backward CGO} such that $$v_{2,\rho}=\psi_{+, \rho}(\theta_{-,\rho}+z_{-, \rho, q_2})$$ with $\norm{z_{-, \rho, q_2}}_{L^2(Q)}$ 
	tending to $0$,  as $\rho\to\infty$.
	Then
	\begin{align*}
		v_{1,\rho}v_{2,\rho}&=\theta_{+, \rho}\theta_{-, \rho}+\theta_{+, \rho} z_{-, \rho, q_2}
		+z_{+, \rho, q_1}\theta_{-, \rho}+z_{+, \rho, q_1}z_{-, \rho, q_2}\\
		&=\varphi_{\rho}(t)e^{-\mbox{i}(x,t)\cdot(\xi,\tau)}+\theta_{+,\rho}z_{-,\rho, q_2}+z_{+, \rho, q_1}\theta_{-,\rho}+z_{+, \rho, q_1}z_{-, \rho, q_2},
	\end{align*}
	where $\varphi_{\rho}(t)=1-\exp(-\rho^{3/4}t)-\exp(-\rho^{3/4}(T-t))+\exp(-\rho^{3/4}T)$.
	Note that $\theta_{+,\rho}$ and $\theta_{-,\rho}$ are bounded with respect to $\rho>0$. 
	Hence, letting $\rho\to+\infty$ in  $\displaystyle\int_Q fv_{1, \rho}v_{2, \rho}\, dxdt=0$, we have that
	\begin{equation}\label{integral id}
		\int_Qf e^{-\mbox{i}(x,t)\cdot(\xi,\tau)}\, dxdt=0.
	\end{equation}
	Therefore, for a fixed $\omega\in\mathbb{S}^{n-1}$, \eqref{integral id} holds in any compact subset of $M$. Clearly, $M$ is an $n$-dimensional subspace of $\mathbb{R}^{n+1}$. Notice that $f$ has compact support as a distribution and its Fourier transform is analytic. The Fourier transform of $f$ is zero in any compact subset of $M$ as shown, and  therefore by changing $\omega\in\mathbb{S}^{n-1}$ in a small conic neighborhood, 
	we can conclude it is zero in $\mathbb{R}^{n+1}$.  This implies $f=0$ in $Q$ as desired.
\end{proof}

In the application of the preceding denseness result with full data, we are able to derive the following global uniqueness result as follows.

\begin{cor}[Global uniqueness with full data]\label{fulldata}
	Let $q_1,q_2 \in L^\infty(Q)$.
	Let $\Lambda_{q_j}$ be the full DN map of the linear heat equation:
	\begin{align}\label{equ forward full data}
		\begin{cases}
			\p_{ t} v_j-\Delta   v_j+q_j v_j=0   &\text{ in }\  Q,\\
			v_j(x,0)=0     &\ \text{ in }\Omega,
		\end{cases}                     
	\end{align}
	for $j=1,2$, respectively. Assume that 
    \begin{align}\label{Same DN full}
    	 \Lambda_{q_1}(f)=\Lambda_{q_2}(f) \text{ on }\Sigma,
    \end{align}
	for any $f\in H^{-1/2,-1/4}(\Sigma)$, then $q_1=q_2$ in $Q$.
	
\end{cor}

\begin{proof}
	This result can be regarded as an application of \cite{CY2018logarithmic}, and we offer the proof for the sake of completeness. Let $\hat{v}$ be a solution to the backward heat equation:
	\begin{align}\label{equ backward full data}
		\begin{cases}
			\p_{ t} \hat v+\Delta  \hat v-q_2 \hat v=0   &\text{ in }\ Q,\\
			\hat v(x,T)=0      &\text{ in }\ \Omega.
		\end{cases}                       
	\end{align}
Subtracting \eqref{equ forward full data} with $j=1,2$, then we have 
\begin{align}\label{equ forward full data difference}
	\begin{cases}
		\p_{ t} \wt v-\Delta   \wt v +q_2 \wt v= (q_2-q_1)v_1   &\text{ in }\  Q,\\
		\wt v(x,0)=0     &\ \text{ in }\Omega,
	\end{cases}                     
\end{align}
where $\wt v=v_1-v_2$ in $Q$. Multiplying \eqref{equ forward full data difference} by the solution $\hat{v}$ of \eqref{equ backward full data}, with the condition \eqref{Same DN full} at hand, it is easy to derive that 
\begin{align}
	\int_Q (q_2-q_1)v_1 \hat{v}\, dxdt=0.
\end{align}
Therefore, by applying \eqref{Lem: denseness}, one can conclude that $q_1=q_2$ in $Q$ as desired.

\end{proof}

\begin{lem}[Global uniqueness with partial data]\label{Lem:CGO}
	Let $\Omega\subset \R^n$ be a bounded domain with $C^\infty$-smooth boundary $\Gamma$.	
	For  any $q_j\in C^{2+\alpha, 1+\frac{\alpha}{2}}(\overline{Q})$   $(j=1, 2)$,   
	assume that $\Lambda^{\mathrm{P}}_{q_j}$ are the partial DN maps of the linear parabolic equation:   
	
	\begin{align}\label{linear term parabolic}
		\begin{cases}
			(\p_t-\Delta  +q_j) u=0  &\text{ in }\ Q,\\
			u= f   &\text{ on }\ \Sigma,\\
			u(x,  0)=0,   &\text{ in }\ \Omega,
		\end{cases}
	\end{align}
	and
	$$\Lambda^{\mathrm{P}}_{  q_1}(f)=\Lambda^{\mathrm{P}}_{ q_2}(f)  \text{ in } \mathcal{V}_- , $$
	for any $f \in C_0^{2+\alpha,1+\alpha/2}(\mathcal{V}_+)$. 
	If $q_1=q_2$ in $\Omega'\times(0,T)$, where $\Omega'$ is an arbitrarily given
	 connected open subset of $\Omega$ with $\Gamma\subset\p\Omega'$,  then
	$$
	q_1=q_2 \text{ in }Q.
	$$
\end{lem}
\begin{proof}
	
	By Proposition \ref{CGO_F}, there is a solution 
	$$v_1(\cdot, \cdot; \rho,  \xi, \tau,\omega)=\psi_{-,\rho}(\theta_{+,\rho}+z_{+, \rho, q_1})\in L^2(0,T;H^1(\Omega))\cap H^1(0,T;H^{-1}(\Omega))$$ 
	to the forward parabolic equation  (\ref{linear term parabolic}) with respect to $q_1$ such that 
	$$
	\lim\limits_{\rho\to\infty}\norm{z_{+,  \rho, q_1}}_{L^2(Q)}=0.
	$$ 
	For $j\in\{1,2\}$, let us define 
	$$
	S_j=\Big\{v\in L^2(0,T;H^1(\Omega))\cap H^1(0,T;H^{-1}(\Omega))\,\Big|\, (\p_t-\Delta  +q_j) 
	v=0  \text{ in } Q, v(x,0)=0 \text{ in }\Omega \Big\},
	$$
    and the map $\mathcal{M}:S_1\to S_2$ is defined by 
	$$\mathcal{M}(v_1)=v_2,$$
	where $v_2$ is the solution to
		\begin{align}
		\begin{cases}
			(\p_t-\Delta  +q_2) v_2=0  &\text{ in }\ Q,\\
			v_2= v_1   &\text{ on }\ \Sigma,\\
			v_2(x,  0)=0,   &\text{ in }\ \Omega.
		\end{cases}
	\end{align}
	By using the trace theorem, $v_1 \big|_{\Sigma}\in L^2(0,T;H^{1/2}(\Gamma)) $ and the map $\mathcal M$ is well-defined.
	Now we have
	\begin{align}\label{111}
		\begin{cases}
			(\p_t-\Delta  +q_2) (v_1-v_2)=(q_2-q_1)v_1  &\text{ in }\ Q,\\
			v_1-v_2= 0   &\text{ on }\ \Sigma,\\
			(v_1-v_2)(x,0)=0 &\text { in }\ \Omega.
		\end{cases}
	\end{align}

	Consider a solution $\hat{v}$ to the backward parabolic equation \eqref{cgo2} of the form that we have constructed in Proposition \ref{CGO_F} with $q=q_2$. 
	Then by Lemma $\ref{appro2}$, there are two sequences of functions 
	$\left\{v^k_1 \right\}^{\infty}_{k=1}$, $\left\{\hat{v}^k\right\}^{\infty}_{k=1}\in 
	C^{2+\alpha, 1+\frac{\alpha}{2}}(\overline{Q})$, such that 
	  $v^k_1$ are solutions to \eqref{app3}, $\hat{v}^k$ are solutions to \eqref{app4}, and  $v^k_1\to v_1$, $\hat{v}^k\to\hat{v}$ in $L^2(Q')$ as $k\to \infty$.  Hence, we have 
	  	\begin{align}
	  	\begin{cases}
	  		(\p_t-\Delta  +q_2) (v^k_1-\mathcal{M}(v_1^k))=(q_2-q_1)v^k_1  &\text{ in }\ Q,\\
	  		v^k_1-\mathcal{M}(v_1^k)= 0   &\text{ on }\ \Sigma,\\
	  		(v^k_1-\mathcal{M}(v_1^k))(x,0)=0 &\text { in }\ \Omega.
	  	\end{cases}
	  \end{align}
	Let $v^k_2=\mathcal{M}(v_1^k)$. Multiplying by the functions $\hat{v}^k$ on the both sides of the above  equation and integration by parts implies
	$$\int_{Q} \LC q_2-q_1 \RC v^k_1\hat{v}^k\ dxdt =\int_{\Sigma}
	\hat{v}^k\partial_\nu (v^k_1-v_2^k)\, dS dt.$$

	Since $\mathcal{U}_{\pm}$ is a neighborhood of $\Gamma_{\pm,\omega_0}$ (recalling $\mathcal{V}_{\pm}=\mathcal{U}_{\pm}\times(0,T)$), there is an $\varepsilon>0$, such that\footnote{We also utilize the same parameter $\varepsilon$ to construct the solution $v_1$.}
	\begin{align*}
		\Big\{x\in\Gamma \, \Big| \, 0<\omega_0\cdot\nu(x)<2\varepsilon\Big\}\times(0,T)&\subset\mathcal{V}_-,\\
		\Big\{x\in\Gamma \,  \Big| \, \omega_0\cdot\nu(x)>-2\varepsilon\Big\}\times(0,T)&\subset\mathcal{V}_+.
	\end{align*}
	Therefore, by choosing 
	$$
	\omega\in \Big\{\omega\in\mathbb{S}^{n-1}\, \Big| \, |\omega-\omega_0|<\varepsilon\Big\},
	$$ we get  that 
	\begin{align*}
		\text{supp } v^k_1\Big|_{\Sigma}&\subset \Big\{x\in\Gamma \,  \Big| \, \omega\cdot\nu(x)\geq -\varepsilon \Big\} \times(0,T) \\ &\subset \Big\{x\in\Gamma \,  \Big| \, \omega_0\cdot\nu(x)>-2\varepsilon\Big\}\times(0,T)\subset\mathcal{V}_+, 		
	\end{align*} 
     and 
     \begin{align*}
     	\Big\{x\in\Gamma \, \Big|\, \omega_0\cdot\nu(x)\geq 2\varepsilon \Big\} \subset \Gamma_{+,\omega,\varepsilon} .
     \end{align*}
    Note that $v_1^k\Big|_{\Sigma}=v_2^k\Big|_{\Sigma}\in  C_0^{2+\alpha,1+\alpha/2}(\mathcal{V}_+)$ and recall  $\hat{v}^k=0$ on $ \Gamma_{+,\omega,\varepsilon}$. Then we have
	\begin{align*}
		&\left|\int_{\Sigma}\hat{v}^k\partial_\nu (v^k_1-v^k_2)\, dSdt\right|\\
		=&\left|\int_{\{\omega_0\cdot\nu\geq2\varepsilon\}}\hat{v}^k\partial_\nu (v^k_1-v^k_2)\,dSdt\right|+	\left|\int_{\{0<\omega_0\cdot\nu<2\varepsilon\}}\hat{v}^k\partial_\nu (v^k_1-v^k_2)\,dSdt\right|\\
		&+\left|\int_{\{\omega_0\cdot\nu\leq 0\}}\hat{v}^k\partial_\nu (v^k_1-v^k_2)\,dSdt\right|\\
		=&0
	\end{align*}
 Then 
 	$$\int_{Q'} \LC q_2-q_1 \RC v^k_1\hat{v}^k\ dxdt+ \int_{Q\backslash Q'} \LC q_2-q_1 \RC v^k_1\hat{v}^k\ dxdt= 0.$$

   Since we assume $q_1=q_2$ in $Q \setminus Q'$, it follows that
   	$$\int_{Q'} \LC q_2-q_1 \RC v^k_1\hat{v}^k\, dxdt=0.$$
	Therefore,  by  similar arguments as in Corollary \ref{fulldata}, letting $\rho\to\infty$, one has that
	\begin{equation*}\label{Q'}
		\int_{Q'}(q_2-q_1) e^{-\mathrm{i}(x,t)\cdot(\xi,\tau)}\, dxdt=0,
	\end{equation*}
    where $\mathrm{i}=\sqrt{-1}$.
	Since $
	\omega\in \Big\{\omega\in\mathbb{S}^{n-1}\, \Big| \, |\omega-\omega_0|<\varepsilon\Big\},
	$  it can be changed in a small conic neighborhood. By by using similar arguments as in Corollary \ref{fulldata}, we have 
	$$
	q_1=q_2 \text{ in } Q
	$$
	as desired.
\end{proof}

\begin{rmk}
	For the full data case, we can use Lemma \ref{appro} to get an approximation of CGO solutions instead of Lemma \ref{appro2}. We do not need to assume $q_1=q_2$ in $\Omega'\times(0,T)$, and we also point out that we cannot apply Corollary \ref{fulldata} to get the result for full data because we need to control the trace of solution on $\Sigma$. In other words, Lemma \ref{appro} is necessary even for the full data case in this paper.
\end{rmk}

\subsection{Proof of Theorem \ref{Main Thm:Simultaneous}}
With Lemma \ref{Lem:CGO} at hand, combining with the higher order linearization method, we are able to prove Theorem \ref{Main Thm:Simultaneous}.

\begin{proof}[Proof of Theorem $\ref{Main Thm:Simultaneous}$]  Let us first remark that the proofs of (a) and (b) in  Theorem \ref{Main Thm:Simultaneous} are similar, so it suffices to show the global uniqueness result with partial data. The  whole  proof is 
	divided into five parts. 
	
	\medskip
	
	{\it Step 1. Initiation} 
	
	\medskip
	
	\noindent Let us introduce the following boundary value 
	\begin{align}\label{small lateral BCs}
		f(x,t;\eps)=\sum_{\ell=1}^M \eps_\ell f_\ell\quad \text{on} \   \Sigma, 
	\end{align}
	where $M\in \N$,  $f_1, \cdots,  f_M \in C^{2+\alpha,1+\alpha/2}_0(\mathcal{V}_+)$  
	and $\eps=(\eps_1, \ldots,\eps_M)$ is a parameter vector in $\mathbb R^M$  with  $|\eps|=\sum\limits_{\ell=1}^{M}  |\epsilon_\ell|$ small enough,  such that  $\left\|\displaystyle\sum_{\ell=1}^M \eps_\ell f_\ell \right\|_{C^{2+\alpha,1+\alpha/2}_0(\overline{\Sigma})}$ is sufficiently small.
	For $j=1,  2$, by the local well-posedenss property in Section \ref{Sec 2}, there exist unique solutions $u_j=u_j(x,t;\eps)\in C^{2+\alpha,1+\alpha/2}(\overline{Q})$ to
	\begin{align}\label{IBVP of simultaneous recovery-eps}
		\begin{cases}
			u_{j,  t}-\Delta u_j+b_j(x,t,u_j)=0  &\text{ in }\ Q,\\
			u_j=\displaystyle\sum_{\ell=1}^M \eps_\ell f_\ell\ &\text{ on }\ \Sigma,\\
			u_j(x, 0)=g_j(x) &\text{ in }\ \Omega,
		\end{cases}
	\end{align}
	where $g_j\in C^{2+\alpha}_0(\Omega)$ with $\norm{g_j}_{C^{2+\alpha}(\Omega)}<\frac{\delta}{2}$ being sufficiently small, and $b_j(x,t,z)$ are admissible coefficients defined in Section \ref{Sec 1}.
    For the sake of convenience, when $\eps=0 $, let $\wt u_j=u_j(\cdot, \cdot; 0)$ be the solutions to 
	\begin{align}\label{IBVP of simultaneous recovery-eps=0}
		\begin{cases}
			\wt u_{j,  t}-\Delta \wt u_j +b_j(x,t, \wt u_j)=0  &\text{ in }\ Q,\\
			\wt u_j=0   &\text{ on }\ \Sigma,\\
			\wt u_j(x, 0)=g_j, &\text{ in }\ \Omega.
		\end{cases}
	\end{align}
	By utilizing the higher order linearization to \eqref{IBVP of simultaneous recovery-eps} around 
	the solution $\wt u_j$ to \eqref{IBVP of simultaneous recovery-eps=0},  we will  determine  information on  $b_j$ for $j=1,2$.
	
	\medskip
	
	{\it Step 2. The first order linearization $(M=1)$}
	
	\medskip
	
	\noindent One can linearize the equation \eqref{IBVP of simultaneous recovery-eps} around $\wt u_j$, where $\wt u_j$ is the solution to \eqref{IBVP of simultaneous recovery-eps=0}, for $j=1,2$. Due to Remark \ref{remark 2.2}, direct computations demonstrate that for $j=1, 2$ and $\ell=M=1$\footnote{In fact, the arguments hold for all $\ell=1,\ldots,M$, and we will use in steps 2-5.},  
	$$
	v_j^{(\ell)}(x,t)=\lim_{\eps \to 0} \frac{u_j(x,t)-\wt u_j(x,t)}{\eps_\ell}
	$$  satisfies the following parabolic equation: 
	\begin{align}\label{first linearization}
		\begin{cases}
			v_{j,  t} ^{(\ell)}-\Delta v_j^{(\ell)}+q_j v_j ^{(\ell)}=0 &\text{ in }\ Q, \\
			v_j ^{(\ell)}= f _\ell& \text{ on }\ \Sigma,\\
			v_j^{(\ell)}(x,  0)=0 &\text{ in }\ \Omega,
		\end{cases}
	\end{align}
	where 
	\[
	q_j (x,t):=b_{j,  u} (x,t,\wt u_j(x,t))\  \text{ in }Q\quad \mbox{ and }\quad q_j\in C^{2+\alpha,1+\frac{\alpha}{2}}(\overline{Q}). 
	\]
	We need to point out that both $\wt u_j$ and $v_j^{(\ell)}$ in \eqref{IBVP of simultaneous recovery-eps=0} and \eqref{first linearization} are still unknown, respectively,  since they solve parabolic equations with unknown coefficients and initial data.  In this  step,   we will  show that 
	\begin{align}\label{claim1}
	q_1 (x,t)=q_2 (x,t)\ \text{ in }  Q.
	\end{align}
	With the same partial DN maps at hand 
	$$\Lambda^{\mathrm{P}}_{b_1,g_1}(f)=\Lambda^{\mathrm{P}}_{b_2,g_2}(f), \quad \text{ for any sufficiently small }f\in C^{2+\alpha,1+\frac{\alpha}{2}}_0(\mathcal{V}_+),
	$$
	such that we have 
	\begin{align}\label{first linearized DN maps agree}
		\begin{array}{rll}
			&v_1^{(\ell)}(x, 0)=v_2^{(\ell)}(x, 0), \quad   \left. v_1^{(\ell)} \right|_{\Sigma}= \left. v_2^{(\ell)}\right|_{\Sigma}, \quad &	\left. \p _\nu v_1^{(\ell)} \right|_{\mathcal{V}_{-}}= \left. \p_\nu  v_2^{(\ell)}\right|_{\mathcal{V}_{-}}, 
		\end{array}
	\end{align} 
	for $\ell=M=1$.

	Now, subtracting \eqref{first linearization} with $j=1,2$, we have 
	\begin{align}\label{first linearization subtraction}
		\begin{cases}
			v^{(\ell)}_{t}- \Delta v^{(\ell)} +q_2 v^{(\ell)}= (q_2 -q_1 )v^{(\ell)}_1   & \text{ in }\ Q, \\
			v^{(\ell)}=0 &\text{ on }\ \Sigma, \\
			v^{(\ell)}(x,0)=0 &\text{ in }\ \Omega,
		\end{cases}
	\end{align}
	where $v^{(\ell)}:=v^{(\ell)}_1-v^{(\ell)}_2$. Let $\tilde  v_2^{(\ell)}$ be a solution to the
	 following backward parabolic equation:
	\begin{align}\label{tilde v_1 equation}
	\begin{cases}
			\tilde  v_{2,t}^{(\ell)}+\Delta \tilde  v_2^{(\ell)} - q_2 \tilde  v_2^{(\ell)}=0 & \text{ in }\ Q, \\
			\tilde v_2(x,T)=0 & \text{ in }\ \Omega.
	\end{cases}
	\end{align}
	Multiplying both sides of  the first equation   in \eqref{first linearization subtraction} by $\tilde  v_2^{(\ell)}$, by  \eqref{first linearized DN maps agree},  an integration by parts yields that  
	\begin{align}\label{integral id of 1st linearized equation1}
		\int_{Q} \LC q_2-q_1 \RC v_1^{(\ell)}\tilde v_2^{(\ell)}\, dxdt =\int_{\Sigma}
		\tilde v_2^{(\ell)}\partial_\nu v_1^{(\ell)}\, dS dt.
	\end{align}
	Moreover, with the condition $b_1=b_2 $ in $\Omega'\times(0,T)\times \R$ at hand, by applying Lemma \ref{Lem:CGO}, one can easily see that the claim \eqref{claim1} holds. Furthermore, as $q_1=q_2$ in $Q$, $v_1^{(\ell)}$ and $v_2^{(\ell)}$ satisfy the same parabolic equation \eqref{first linearization}, by the uniqueness of solutions,  we obtain  that
	\begin{align}\label{uniqueness for solutions of the first linearized equation}
	v^{(\ell)}:= v_1^{(\ell)}=v_2^{(\ell)} \text{ in }Q.
	\end{align}

	\medskip
	
	{\it Step 3. The second order linearization  $(M=2)$}
	
	\medskip

	\noindent  For the second linearization ($m=2$),  one can differentiate  (\ref{IBVP of simultaneous recovery-eps}) 
	with respect to different parameters $\eps_1$ and $\eps_2$. 
	A direct computation shows that $w^{(2)}_j$ $(j=1, 2)$ satisfy
	\begin{align}\label{second linearization}
		\begin{cases}
			w_{j, t} ^{(2 )}-\Delta w_j^{(2)}+qw_j^{(2)} +b_{j, uu} (x,t,\wt u_j)v^{(1)}v^{(2)}=0 &\text{ in }\ Q, \\
			w_j^{(2)} = 0 & \text{ on }\ \Sigma,\\
			w_j^{(2)}(x, 0)=0 &\text{ in }\ \Omega, 
		\end{cases}
	\end{align}
	where $q=q_1=q_2$, $b_{j, uu}(\cdot, \cdot, \tilde u_j)\in C^{2+\alpha,1+\frac{\alpha}{2}}(\overline{Q})$ 
	and $v^{(1)}, v^{(2)}\in C^{2+\alpha,1+\frac{\alpha}{2}}(\overline{Q})$ satisfy 
	\begin{align*}
		\begin{cases}
			v_{t} ^{(\ell)}-\Delta v^{(\ell)}+q(x,t)v^{(\ell)}=0 &\text{ in }\ Q, \\
			v^{(\ell)}= f_\ell& \text{ on }\ \Sigma,\\
			v^{(\ell)}(x,  0)=0 &\text{ in }\ \Omega,
		\end{cases}
	\end{align*}
	here  $f_1$ and $f_2$ can be arbitrarily chosen.

	Next, we will prove that 
	\begin{align}\label{claim2}
		b_{1,  uu} (x,t,\wt u_1(x, t))=b_{2, uu}(x,t,\wt u_2(x, t)) \quad {in}\quad Q.
	\end{align}
	With the same DN map at hand, by differentiating $\eps_1$ and $\eps_2$, we have 
	\begin{align}\label{second linearized DN maps agree}
		\begin{array}{rll}
			w_1^{(2)}(x, 0)=w_2^{(2)}(x, 0), \quad  w_1^{(2)} \Big|_{\Sigma}=w_2^{(2)}\Big|_{\Sigma}, \quad &	\p _\nu w_1^{(2)} \Big|_{\mathcal{V}_{-}}=\p_\nu  w_2^{(2)}\Big|_{\mathcal{V}_{-}   }.
		\end{array}
	\end{align}

	Let $v^{(0)}$ be any solution to the backward parabolic equation: 
	\begin{eqnarray}\label{AA}
		\left\{
		\begin{array}{ll}
			v_{t}^{(0)}+\Delta   v^{(0)} - q  v^{(0)}=0 &\text{ in } Q,\\[2mm]
			v^{(0)}(x, T)=0 &\text{ in }\Omega.
		\end{array}
		\right.
	\end{eqnarray}
	By subtracting the equations \eqref{second linearization} associated to $j=1,  2$, an integration by parts yields
	\begin{align}\label{second integral id}
		\begin{split}
			\int_{Q} \Big[b_{1, uu} (x,t,\wt u_1(x, t))-b_{2, uu}(x,t,\wt u_2(x, t))\Big] v^{(0)}  v^{(1)}v^{(2)}  \, dxdt =0
		\end{split}
	\end{align}
	 We next choose a nonzero boundary data $f_2$ such that $f_2\geq 0$ on  $\Sigma$  and $f_2>0$ on  $D_t\times(0, T)$, where  $D_t\subset \Gamma$ is a relative open subset for  any $t\in (0, T)$. Via the condition $f_2=v^{(2)}|_{\Sigma}\in L^\infty(\Sigma)$ at any time $t\in  (0, T)$, by applying the maximum principle for parabolic equation (for example, see \cite[Chapter 7]{evans1998partial} or Appendix \ref{Sec: Appendix}), we have a bounded positive solution $v^{(2)}$ in $Q$. 
	 Now, by selecting $v^{(1)}$ and $v^{(0)}$ as the CGO solutions of forward and backward parabolic equations, via Corollary \ref{fulldata},  we get 
	 \[
	  \big[b_{1, uu} (x,t,\wt u_1(x, t))-b_{2, uu}(x,t,\wt u_2(x, t))\big] v^{(2)}=0 \text{ in }Q.
	 \]
	 With the positivity of $v^{(2)}$ in $Q$ at hand, we have (\ref{claim2}) as desired. 
	Furthermore, by  the uniqueness of solutions to \eqref{second linearization}, one can immediately obtain 
	\[
	w_1^{(2)}= w_2^{(2)} \quad\text{ in }Q.
	\]

	\medskip
	
	{\it Step 4. The higher order linearization $(M>2)$}
	
	\medskip
	
	\noindent  By utilizing the higher order linearization with the induction hypothesis, we are able to find $M$-th order derivative of \eqref{IBVP of simultaneous recovery-eps}  and prove that
	\begin{align}\label{claim3}
		\p_u^M b_1 (x,t,  \wt u_1(x, t))=\p_u^M b_2 (x,t, \wt u_2(x, t))\quad \text{ in }Q,
	\end{align}
	for any $M=3, 4, \cdots$.
	Let us first  assume that 
	$$
	\p_u^k b_1 (x,t, \wt u_1(x, t))=\p_u^k b_2 (x,t, \wt u_2(x, t))\  \text{ in }Q, \  \text{ for any }k=1,\dots, M-1.
	$$
	Similar to previous steps, we differentiate  \eqref{IBVP of simultaneous recovery-eps} with respect to $\eps_1,\ldots, \eps_{M-1}$ and $\eps_{M}$, then we have 
	\begin{align*}
		\int_{Q}\Big[\p_u ^M b_1 (x,t, \wt u_1(x, t)) - \p_u^M b_2 (x,t, \wt u_2(x, t))\Big]  v^{(0)}v^{(1)} \cdots v^{(M)}\, dxdt =0,
	\end{align*}
	where $v^{(0)}$ is the solution to the backward parabolic equation \eqref{AA}, and $v^{(\ell)}$ $(\ell=1,2, \cdots, M)$ are solutions to the forward parabolic equation \eqref{first linearization}.
	Similar to Step 3, 	let us choose $v^{(0)}$ and $v^{(1)}$ as CGO solutions, and $v^{(2)},\ldots, v^{(M) }$ are bounded positive solutions in $Q$ 
	\begin{equation}\label{RRR}
		\p_u ^M b_1 (x,t, \wt u_1(x, t))=\p_u^M b_2 (x,t, \wt u_2(x, t))\mbox{ in }Q, \quad \text{ for any }M\in \N.
	\end{equation}

	\medskip
	
	{\it Step 5. The determination of  initial data and coefficients}
	
	\medskip
	
	\noindent Recall that  $\wt u_j$ $(j=1, 2)$ are the solutions to the semilinear parabolic equation:
	\begin{align*}
		\begin{cases}
			\wt u_{j,  t}-\Delta \wt u_j +b_j(x,t, \wt u_j)=0  &\text{ in }\ Q,\\
			\wt u_j=0   &\text{ on }\ \Sigma,\\
			\wt u_j(x, 0)=g_j, &\text{ in }\ \Omega.
		\end{cases}
	\end{align*}
	As in the proof of \cite[Theorem 1.3]{LLL2021determining}, by the admissible property of $b_1$ and $b_2$,
	\begin{eqnarray}\label{LLL}
		\begin{array}{ll}
			&b_1(x, t, \wt u_1(x, t))-b_2(x, t, \wt u_2(x, t))\\[2mm]
			=&\displaystyle\sum\limits_{k=1}^\infty \frac{\partial_u^k  b_2(x, t, \wt u_2(x,  t))}{k!}\Big[-\wt u_2(x, t)\Big]^k
			-\sum\limits_{k=1}^\infty \frac{\partial_u^k  b_1(x, t, \wt u_1(x,  t))}{k!}\Big[-\wt u_1(x, t)\Big]^k\\
			=&\displaystyle\sum\limits_{k=1}^\infty \frac{\partial_u^k  b_1(x, t, \wt u_1(x,  t))(-1)^k}{k!}\Big\{\Big[\wt u_2(x, t)\Big]^k-\Big[\wt u_1(x, t)\Big]^k\Big\}.
		\end{array}
	\end{eqnarray}
	Since both $\wt u_1$
	and $\wt u_2$ are bounded, set $R=\|\wt u_1\|_{L^\infty(Q)}+\|\wt u_2\|_{L^\infty(Q)}$. Then, for any $L>0$ and $(x, t)\in Q$, 
	\begin{align*}
		&\left|\frac{b_1(x, t, \wt u_1(x, t))-b_2(x, t,\wt u_2(x, t))}{\wt u_1(x, t)-\wt u_2(x, t)}\right|\\
		=&\left|\sum\limits_{k=1}^\infty \frac{\partial_u^k  b_1(x, t, \wt u_1 (x,  t))}{k!}(-1)^{k+1}\Big\{\Big[\wt u_1(x, t)\Big]^{k-1}+\Big[\wt u_1(x, t)\Big]^{k-2}\wt u_2(x, t)+\cdots\right.\\
		&\quad \left.+\wt u_1 (x, t)\Big[\wt u_2(x, t)\Big]^{k-2}+\Big[\wt u_2(x, t)\Big]^{k-1}\Big\}\right|\\
		\leq &\sum\limits_{k=1}^\infty \left|\partial_u^k  b_1(x, t, \wt u_1(x,  t))\right|
		\frac{R^{k-1}}{(k-1)!}  \\
		\leq &\sum\limits_{k=1}^\infty \frac{k R^{k-1} }{L^k} \sup\limits_{|z-\wt u_1(x, t)|=L}  
		|b_1(x, t, z)|.
	\end{align*}
	Choose  $L=2(R+1)$. By the  admissibility of $b_1$ and $b_2$, 
	$$
	G(\cdot, \cdot)=\frac{b_1(\cdot, \cdot, \wt u_1(\cdot, \cdot))-b_2(\cdot, \cdot, \wt u_2(\cdot, \cdot))}{\wt u_1(\cdot, \cdot)-\wt u_2(\cdot, \cdot)}
	\in L^\infty(Q).
	$$
	Set $w=\wt u_1-\wt u_2$. It is easy to see that
	\begin{align*}
		\begin{cases}
			w_{t}-\Delta w +Gw=0  &\text{ in }\ Q,\\
			w=0   &\text{ on }\ \Sigma,\\
			w(x, 0)=g_1-g_2   &\text{ in }\ \Omega.
		\end{cases}
	\end{align*}
	By $\Lambda_{b_1,g_1}(0)=\Lambda_{b_2, g_2}(0)$
	and Lemma \ref{lemma3}, we have 
	$$
	g_1=g_2 \text{ in }\Omega \quad \mbox{ and }\quad \wt u_1=\wt u_2 \text{ in }\ Q.
	$$
	By (\ref{LLL}),  
	$$
	b_1(x, t, \wt u_1(x, t))=b_2(x, t, \wt u_2(x, t)) \quad\mbox{ in }\ Q.  
	$$
	In addition,  note that for $j=1, 2$ and any $(x, t, z)\in Q\times\mathbb R$,
	\begin{align*}
		&&b_j(x, t, z)=b_j(x, t, \wt u_j(x, t))
		+\sum\limits_{k=1}^{\infty} 
		\frac{\partial_u^k b_j(x, t, \wt u_j(x, t))}{k!}\LC z-\wt u_j(x, t)\RC^k,
	\end{align*}
	which implies that $b_1(x, t, z)=b_2(x, t, z)$ in  $Q\times\mathbb R$. This proves the assertion.
\end{proof}

\subsection{Proof of Theorem \ref{Main Thm:Simultaneous linear}}

Similar to the proof of Theorem \ref{Main Thm:Simultaneous}, we are ready to prove Theorem \ref{Main Thm:Simultaneous linear}.

\begin{proof}[Proof of Theorem \ref{Main Thm:Simultaneous linear}]
The argument is similar to the proof of Theorem \ref{Main Thm:Simultaneous}, and we prove this result with the full data. Let us divide the proof into two steps. 

	\medskip

{\it Step 1. Unique determination of coefficients}

\medskip

\noindent Let $u_j=u_j(x,t)$ be the solution to
\begin{align*}
	\begin{cases}
		u_{j,  t}-\Delta u_j+q_j u_j=0  &\text{ in }\ Q,\\
		u_j=f &\text{ on }\ \Sigma,\\
		u_j(x, 0)=g_j(x) &\text{ in }\ \Omega,
	\end{cases}
\end{align*}
and let $\wt u_j=\wt u_j(x,t)$ be the solution to  
\begin{align}\label{IBVP of simultaneous recovery-linear2}
	\begin{cases}
		\wt u_{j,  t}-\Delta \wt u_j+q_j \wt u_j=0  &\text{ in }\ Q,\\
		\wt u_j=0 &\text{ on }\ \Sigma,\\
		\wt u_j(x, 0)=g_j(x) &\text{ in }\ \Omega,
	\end{cases}
\end{align}
for $j=1,2$. With the same DN maps on the lateral boundary at hand, we have 
\begin{align}\label{same DNs}
	\p_\nu u_1 =\p_\nu u_2 \quad \text{ and }\quad \p_\nu \wt u_1 =\p_\nu \wt u_2 \quad \text{ on }\quad \Sigma.
\end{align}
We next consider $v_j :=u_j -\wt u_j$ for $j=1,2$, then $v_j$ is the solution of
\begin{align}\label{IBVP of simultaneous recovery-linear3}
	\begin{cases}
		v_{j,  t}-\Delta v_j+q_j v_j=0  &\text{ in }\ Q,\\
		v_j=f &\text{ on }\ \Sigma,\\
		v_j(x, 0)=0 &\text{ in }\ \Omega.
	\end{cases}
\end{align}

Subtracting \eqref{IBVP of simultaneous recovery-linear3} with respect to $j=1,2$, we get 
\begin{align}\label{IBVP of simultaneous recovery-linear4}
	\begin{cases}
		v_{  t}-\Delta v+q_2 v= (q_2-q_1)v_1 &\text{ in }\ Q,\\
		v=\p_\nu v=0 &\text{ on }\ \Sigma,\\
		v(x, 0)=0 &\text{ in }\ \Omega,
	\end{cases}
\end{align}
where $v=v_1-v_2$. Moreover, via the condition \eqref{same DNs}, we have $\p_\nu v=0$ on $\Sigma$.
On the other hand, let $\wt v_2$ be a solution to the backward parabolic equation 
\begin{align*}
	\begin{cases}
		\wt v_{2,t}+\Delta \wt v_2 +q_2\wt v_2 =0 & \text{ in }\ Q,\\
		\wt v_2(x,T)=0 &\text{ in }\ \Omega.
	\end{cases}
\end{align*}
Multiplying \eqref{IBVP of simultaneous recovery-linear4} by the function $\wt v_2$, an integration by parts yields that 
\begin{align}
	\int_{Q}(q_2-q_1)v_1 \wt v_2 \, dxdt=0.
\end{align}
By applying the global uniqueness result with full data (Corollary \ref{fulldata}), then we have $q_1=q_2$ as desired.

	\medskip

{\it Step 2. Unique determination of initial data}

\medskip

\noindent Recalling that $\wt u_j$ is the solution of \eqref{IBVP of simultaneous recovery-linear2}, by using the uniqueness $q_1=q_2$, we can subtract \eqref{IBVP of simultaneous recovery-linear2} with respect to $j=1,2$, then we obtain 
 \begin{align}\label{IBVP of simultaneous recovery-linear5}
 	\begin{cases}
 		 u_{  t}-\Delta u+q  u=0  &\text{ in }\ Q,\\
 		 u=0 &\text{ on }\ \Sigma,\\
 		 u(x, 0)=g_1-g_2 &\text{ in }\ \Omega,
 	\end{cases}
 \end{align}
where $q=q_1-q_2$ and $u= \wt u_1-\wt u_2$. Via the condition \eqref{same DNs} again, we have $\p_\nu u=0 $ on $\Sigma$. Finally, by applying the quantitative stability estimate \eqref{Stability estimate in Thm 1}, we can obtain the uniqueness of the initial data $g_1=g_2$ in $\Omega$. This proves the assertion.
\end{proof}

\begin{rmk}
	One can find that when the initial and boundary data are small enough, Theorem \ref{Main Thm:Simultaneous linear} can be regarded as a corollary of Theorem \ref{Main Thm:Simultaneous}, where we can simply take $b_j(x,t,u):=q_j(x,t)u$ for $j=1,2$. In order to distinguish the statements of Theorems \ref{Main Thm:Simultaneous} and \ref{Main Thm:Simultaneous linear}, we provide two complete proofs of both theorems.
\end{rmk}

\appendix

\section{Auxiliary results}\label{Sec: Appendix}

In the end of this paper, for the sake of self-containedness, we review some properties for linear hear equations, which were used in our proofs.

\subsection{Complex geometrical optics solutions}

We first prove a density result for the product of solutions to forward and backward parabolic equations in $L^1(Q)$. It  depends on the construction of CGO solutions, which vanish at initial or final time. They were constructed in \cite{CK2018determination}, and we summarize the results as the following propositions. To make the explanation clear, we split the 
procedure  into two parts.

For any $\rho>0$, we define
\begin{align*}
 \begin{cases}
 		\psi_{+,\rho}(x, t)=\exp(-(\rho\omega\cdot x+\rho^2t)), \\
 		 \psi_{-,\rho}(x, t)=\exp(\rho\omega\cdot x+\rho^2t),
 \end{cases}
\end{align*}
and 
\begin{align*}
	\begin{cases}
	   \theta_{+,\rho}(x, t; \xi,  \tau)=\LC 1-\exp(-\rho^{3/4}t)\RC  \exp(-\mbox{i}(x,t)\cdot(\xi,\tau)) ,\\
	   \theta_{-, \rho}(x,t)=1-\exp(-\rho^{3/4}(T-t)),
	\end{cases}
\end{align*}
where $\xi\in\mathbb{R}^n $ with $\xi\cdot\omega=0$ and $\tau\in\mathbb{R}$.


The following proposition was demonstrated in \cite[Propositions 4.3, 4.4]{CK2018determination}, and we state the result without proofs for the sake of convenience.

\begin{prop}\label{CGO_F}
	Let $m,\varepsilon>0$ and $\omega\in\mathbb{S}^{n-1}$. There is a positive constant $C$,   depending only on $Q, m$ and $\varepsilon$, such that for any $q
	\in \Big\{q\in L^{\infty}(Q)\  \Big| \  \norm{q}_{L^{\infty}(Q)}<m\Big\}$. Then we have 
	\begin{itemize}
			\item[(a)] There exists a CGO solution $u \in  L^2(0,T;H^1(\Omega))\cap H^1(0,T;H^{-1}(\Omega))$ to the forward parabolic equation 
		\begin{align}
			\begin{cases}
				(\p_{t}-\Delta +q)u=0   &\text{ in }\  Q,\\
				u(x,0)=0     & \text{ in }\ \Omega,
			\end{cases}                     
		\end{align}
		of the form
		$$u(\cdot, \cdot; \rho,  \xi, \tau)=\psi_{-,\rho}(\theta_{+,\rho}+z_{+, \rho, q}),$$
		where $z_{+, \rho, q}\in  L^2(0,T;H^1(\Omega))\cap H^1(0,T;H^{-1}(\Omega)) $
		$$\lim\limits_{\rho\to\infty}\norm{z_{+,  \rho, q}}_{L^2(Q)}=0$$
			\item[(b)] There exists a CGO solution $u\in L^2(0,T;H^1(\Omega))\cap H^1(0,T;H^{-1}(\Omega))$ to the backward parabolic equation 
		\begin{align}\label{cgo2}
			\begin{cases}
				(-\p_{t}-\Delta +q) u=0   &\text{ in }\  Q,\\
				u=0                   &\text{ on } \  \Gamma_{+,\omega,\varepsilon}\times(0,T),\\
				u(x,T)=0     & \text{ in }\ \Omega,
			\end{cases}                     
		\end{align}
		of the form 
		$$u(\cdot, \cdot; \rho)=\psi_{+, \rho}(\theta_{-,\rho}+z_{-, \rho, q}),$$
		where $z_{-, \rho, q}\in L^2(0,T;H^1(\Omega))\cap H^1(0,T;H^{-1}(\Omega))$ and it satisfies the decay condition:
	$$\lim\limits_{\rho\to\infty}\norm{z_{-,  \rho, q}}_{L^2(Q)}=0$$
	\end{itemize}
	
\end{prop}

\subsection{Maximum principle}

Finally, let us show the maximum principle for a linear parabolic equation.

\begin{lem}[Strong maximum principle]
	Let $\Omega\subset \R^n$ be a bounded  domain with smooth boundary $\Gamma$ for $n \in \N$. Let $q\in 
	C(\overline{Q})$ and $v\in C^{2,  1}(Q)\cap  C(\overline{Q})$ be a solution to
	\begin{align}\label{eq max principle}
		\begin{cases}
			v_t -\Delta v+q  v=0 &\text{ in }\ Q,\\
			v=f &\text{ on }\ \Sigma, \\
			v(x,0)=0 &\text{ in }\ \Omega.
		\end{cases}
	\end{align}
Suppose that $f\geq 0$ on  $\Sigma$  and $f>0$ on  $D_t\times(0, T)$ 
with $D_t\subset \Gamma$ being a relative open subset for  any $t\in (0, T)$, then $v>0$ in $Q$.   \end{lem}

\begin{proof}
	Without loss of generality,   we assume that $q\geq 0$  in $Q$. Otherwise, let	
	$u=e^{-\lambda t}v$, where $\lambda>0$ is a sufficiently large positive parameter. 
	If  there exists a pair $(x_0, t_0)\in Q$,  such that
	$v(x_0, t_0)=0$.  Then by \cite[Chapter 7]{evans1998partial}, 
	$v\equiv  0$  in  $\Omega\times(0, t_0)$.  It contradicts with 
	the fact that $f>0$ on $D_t\times(0,  t_0)$. Hence, $v>0$ in $Q$.

	\end{proof}

\vskip0.5cm

\noindent\textbf{Acknowledgment.} The work of Y.-H. Lin is partially supported by the Ministry of Science and Technology Taiwan, under the Columbus Program: MOST-110-2636-M-009-007. The work of H. Liu is supported by a startup fund from City University of Hong Kong and the Hong Kong RGC General Research Funds (projects 12301420, 12302919 and 12301218).
The  work of  X. Liu is partially supported  by NSF of China under  grants 11871142 and 11971320.

\bibliographystyle{alpha}

\bibliography{ref}

\end{document}